\documentclass[10pt]{article}
\usepackage{latexsym, amsmath, amssymb, amsthm}
\usepackage{graphicx, tikz, setspace, authblk, extarrows}
\usepackage[top=1in, bottom=1in, left=1in, right=1in]{geometry}

\newtheorem{theorem}{Theorem}[section]
\newtheorem{lemma}[theorem]{Lemma}

\newtheorem{corollary}[theorem]{Corollary}
\newtheorem{remark}[theorem]{Remark}
\newtheorem{definition}[theorem]{Definition}
\numberwithin{equation}{section}

\def\p{\partial}
\def\ora{\overrightarrow}	
		\def\ol{\overline}		
\def\m{\mathbb}		\def\t{\tilde}	\def\wt{\widetilde}	
\def\O{\Omega}  \def\lam{\lambda}  \def\eps{\epsilon}  
  \def\v{\varepsilon}  
\def\chifcn{\text{{\bf I}}}
\def\be{\begin{equation}}     \def\ee{\end{equation}}
\def\bes{\begin{equation*}}		\def\ees{\end{equation*}}

\title{Blow-up problems for the heat equation with a local nonlinear Neumann boundary condition}
\author{Xin Yang\thanks{Email: yangxin1@msu.edu},\; Zhengfang Zhou\thanks{Email: zfzhou@math.msu.edu}}
\affil{Department of Mathematics, Michigan State University,\\ East Lansing, Michigan, USA}
\date{}

\begin{document}
\maketitle

\begin{abstract}
This paper estimates the blow-up time for the heat equation $u_t=\Delta u$ with a local nonlinear Neumann boundary condition: The normal derivative $\partial u/\partial n=u^{q}$ on $\Gamma_1$, one piece of the boundary, while on the rest part of the boundary, $\partial u/\partial n=0$. 
The motivation of the study is the partial damage to the insulation on the surface of space shuttles caused by high speed flying subjects. We prove the solution blows up in finite time and estimate both upper and lower bounds of the blow-up time in terms of the area of $\Gamma_1$. In many other work, they need the convexity of the domain $\O$ and only consider the problem with $\Gamma_1=\p\O$. In this paper, we remove the convexity condition and only require $\p\O$ to be $C^{2}$. In addition, we deal with the local nonlinearity, namely $\Gamma_1$ can be just part of $\p\O$. 
\end{abstract}

\bigskip



\section{Introduction and Notations}  
\label{Sec, introduction}
In this paper, $\Omega$ is assumed to be a bounded open set in $\mathbb{R}^{n}$ ($n\geq 2$) with $\partial\Omega\in C^{2}$, $\Gamma_1$ and $\Gamma_2$ are two disjoint open subsets of $\p\O$ with
$\ol{\Gamma}_1\cup\ol{\Gamma}_2=\p\O$, $\widetilde{\Gamma}\triangleq \ol{\Gamma}_1\cap\ol{\Gamma}_2$ is $C^{1}$ when being regarded as $\p\Gamma_1$ or $\p\Gamma_2$. We study the heat equation with a local nonlinear Neumann boundary condition:
\begin{equation}\label{Prob, Simple Model}
\left\{\begin{array}{lll}
u_{t}(x,t)=\Delta u(x,t) &\text{in}& \Omega\times (0,T],\\
\frac{\partial u}{\partial n}(x,t)=u^{q}(x,t) &\text{on}& \Gamma_1\times (0,T],\\
\frac{\partial u}{\partial n}(x,t)=0 &\text{on}& \Gamma_2\times (0,T],\\
u(x,0)=u_0(x) &\text{in}& \Omega,
\end{array}\right.
\end{equation}
where $q>1$, $u_0\in C^{1}(\ol{\O})$, $u_0(x)\geq 0$ and $u_0(x)\not\equiv 0$. The normal derivative on the boundary is defined as following: for any $x\in\p\O, 0<t\leq T$,
\be\label{Normal Deri. Def.}
\frac{\p u}{\p n}(x,t)\triangleq  
\lim_{h\rightarrow 0^{+}} Du(x_h,t)\cdot\ora{n}(x)\,\,\,\text{as long as this limit exists,}\ee
where $\ora{n}(x)$ denotes the exterior unit normal vector at $x$ and $x_h\triangleq x-h\ora{n}(x)$ for $x\in\p\O$. $\p\O$ being $C^2$ ensures that $x_h$ belongs to $\O$ when $h$ is positive and sufficiently small. 

Our work is partially motivated by the Space Shuttle Columbia disaster in 2003. When the space shuttle was launched, a piece of foam broke off from its external tank and struck the left wing causing the insulation there damaged. As a result, the shuttle disintegrated during its reentry to the atmosphere due to the enormous heat generated near the damaged part. Based on this, we start to establish the math model. 

\begin{figure}[!ht]
\centering
\begin{tikzpicture}[scale=0.8]
\begin{large}
\draw (-5,1/2)-- (0,1/2);
\draw (0,1/2)--(5/2,7/2);
\draw [domain=5/2:3] plot ({\x},{7/2+1/16-(\x-11/4)^2});
\draw (3,7/2)--(3,1/2);
\draw (3,1/2)--(4,1/2);
\draw (-5,-1/2)-- (0,-1/2);
\draw (0,-1/2)--(5/2,-7/2);
\draw [domain=5/2:3] plot ({\x},{-7/2-1/16+(\x-11/4)^2});
\draw (3,-7/2)--(3,-1/2);
\draw (3,-1/2)--(4,-1/2);
\draw [domain=90:270] plot ({cos(\x)-5},{1/2*sin(\x)});
\draw [dashed] [domain=0:360] plot({-5+1/8*cos(\x)},{1/2*sin(\x)});
\draw [dashed] [domain=0:360] plot({4+1/8*cos(\x)},{1/2*sin(\x)});
\path (5/4-1/8,-7/8) coordinate (A);
\draw (A) node [above] {$u$};
\draw [color=green] [domain=110:150] plot({5/4+3/4*cos(\x)},{-2+3/4*sin(\x)});
\draw [color=green] [domain=290:330] plot({5/4+3/4*cos(\x)},{-2+3/4*sin(\x)});
\path (5/4,-2+1/4) coordinate (B);
\draw (B) node [right] {$\Gamma_1$};
\path (5/4+1/4,2-1/4) coordinate (D);
\draw (D) node [below] {$\Omega$};
\path (-2,1/2) coordinate (F);
\draw (F) node [above] {$\Gamma_2$};
\path (1/2,-5/2) coordinate (E);
\draw (E) node {$\tilde{u}$};
\end{large}
\end{tikzpicture}
\caption{Model}
\label{Fig, model}
\end{figure}

In Figure \ref{Fig, model}, $\tilde{u}$ represents the outside temperature of the space shuttle and $u$ denotes the inside temperature. When the space shuttle reentered the atmosphere, it compressed the air at a very high speed. During this process, it caused many chemical reactions which produced enormous radiative heat flux. This was the main source of the heat transfer through the broken part on the left wing. In Physics, the radiation heat flux is proportional to the fourth power of the difference between the temperatures. In addition, to simplify the model, we assume $\tilde{u}=F(u)$ is an increasing function of $u$ and treat it as a polynomial, say $u^{m}$ for some $m>1$. Thus on the broken part $\Gamma_1$, we have
$$\frac{\partial u}{\partial n}\sim (\tilde{u}-u)^{4}=(F(u)-u)^{4}\sim(u^{m}-u)^{4}\sim u^{q},$$
for $q=4m>1$. On $\Gamma_2$, one has $\frac{\partial u}{\partial n}=0$, since the insulation there are intact. Inside the space shuttle, we assume it satisfies the heat equation. Thus, the realistic problem is modeled as (\ref{Prob, Simple Model}).

The next thing is to make sense of the solution such that it exists and unique. For any $T>0$, we define \[\mathcal{A}_{T} = C^{2,1}(\O\times(0,T])\cap C(\overline{\O}\times[0,T])\]
and $$\mathcal{B}_T=\{g:(\Gamma_1\cup\Gamma_2)\times(0,T]\rightarrow\m{R}\,\big|\,g\in UC\big(\Gamma_1\times(0,T]\big) \text{\;and\;} g\in UC\big(\Gamma_2\times(0,T]\big)\},$$ where $UC$ refers to uniformly continuous function spaces. From the definition of $\mathcal{B}_{T}$, we know for any $ g\in\mathcal{B}_T$, when it is restricted to $\Gamma_i\times(0,T]$ ($i=1$ or $2$), it has a unique continuous extension to $\overline{\Gamma}_i\times[0,T]$ and we use $g|_{\Gamma_{i}\times(0,T]}\in C(\overline{\Gamma}_i\times[0,T])$ to denote this extension. Moreover, when there is no ambiguity, we just write $g_i\triangleq g|_{\Gamma_{i}\times(0,T]}$ for convenience. But one should notice that $g$ may not extend to a continuous function on $\p\O\times(0,T]$, since it can have a jump between $\Gamma_1$ and $\Gamma_2$. Finally, we endow $\mathcal{B}_T$ with $L^{\infty}\big((\Gamma_1\cup\Gamma_2)\times(0,T]\big)$ norm such that it becomes a Banach space.

The solution to (\ref{Prob, Simple Model}) is understood in the following way. 
\begin{definition}\label{Def, soln to simple model}
For any $T>0$, a solution to (\ref{Prob, Simple Model}) on $\ol{\O}\times[0,T]$ means a function $u$ in $\mathcal{A}_{T}$ that satisfies (\ref{Prob, Simple Model}) pointwise and moreover, for any $(x,t)\in\wt{\Gamma}\times(0,T]$, $\frac{\p u}{\p n}(x,t)$ exists and
\be\label{interface bdry deri. for simple model}
\frac{\p u}{\p n}(x,t)=\frac{1}{2}\,u^{q}(x,t). \ee
\end{definition}

Here (\ref{interface bdry deri. for simple model}) is a technical requirement which ensures the uniqueness of the solution, see the proofs in Lemma \ref{Lemma, weak comparison for linear prob.}, Corollary \ref{Cor, comparison and uniqueness for linear prob. of simple model} and Theorem \ref{Thm, nonlinear comparison and uniqueness}. We will see in Section \ref{Sec, upper bound} that the solution to (\ref{Prob, Simple Model}) always blows up in finite time, so we would like to study the maximal time the solution can exist.

\begin{definition}\label{Def, maximal sol'n, Target Prob}
We call \[T^{*}\triangleq \sup\{T\geq 0:\,\text{there exsits a solution to (\ref{Prob, Simple Model}) on}\,\,\, \ol{\O}\times[0,T]\}\]
to be the maximal existence time for (\ref{Prob, Simple Model}). Moreover, a function $u^{*}\in C^{2,1}\big(\O\times(0,T^{*})\big)\cap C\big(\ol{\O}\times[0,T^{*})\big)$ is called a maximal solution if it solves (\ref{Prob, Simple Model}) on $\ol{\O}\times[0,T]$ for any $T\in(0,T^{*})$.
\end{definition}
 
In this paper, we denote $M_0=\max\limits_{\ol{\O}}u_0$ and write $|\Gamma_1|$ to represent the area of $\Gamma_1$, that is $$|\Gamma_1|=\int_{\Gamma_1}\,dS.$$ 
$\Phi:\m{R}^{n}\times(0,\infty)\rightarrow\m{R}$ represents the fundamental solution of the heat equation, that is 
\be\label{Fund. Soln.}
\Phi(x,t)\triangleq \frac{1}{(4\pi t)^{n/2}}\,e^{-\frac{|x|^2}{4t}}, \quad\forall\, (x,t)\in \m{R}^{n}\times(0,\infty). \ee
We will show the local existence and uniqueness of the solution to (\ref{Prob, Simple Model}). Moreover, both upper and lower bounds for the maximal existence time $T^{*}$ will be given. The main results of this paper are as following.

\begin{theorem}\label{Thm, fundamental theorem}
The maximal existence time $T^{*}$ for (\ref{Prob, Simple Model}) is positive and there exists a unique maximal solution $u^{*}\in C^{2,1}\big(\O\times(0,T^{*})\big)\cap C\big(\ol{\O}\times[0,T^{*})\big)$ to (\ref{Prob, Simple Model}). Moreover, $u^{*}(x,t)>0$ for any $(x,t)\in\ol{\O}\times(0,T^{*})$. 
\end{theorem}

\begin{theorem}\label{Thm, upper bound of blow-up time}
Suppose $T^{*}$ is the maximal existence time for (\ref{Prob, Simple Model}), then $T^{*}<\infty$ and 
\[\sup_{(x,t)\in\ol{\O}\times[0,T^{*})}|u^{*}(x,t)|=\infty.\] In addition, if
$\min\limits_{x\in\ol{\O}}u_{0}(x)>0$, then 
\be\label{upper bound}
T^{*}\leq \frac{1}{(q-1)|\Gamma_1|}\int_{\O}u_0^{1-q}(x)\,dx. \ee
\end{theorem}

\begin{theorem}\label{Thm, lower bound of blow-up time}
Suppose $T^{*}$ is the maximal existence time for (\ref{Prob, Simple Model}), then  
\be\label{lower bound}
T^{*}\geq C^{-\frac{2}{n+2}}\bigg[\ln \Big(|\Gamma_1|^{-1}\Big)-(n+2)(q-1)\ln M_0-\ln(q-1)-\ln C\bigg]^{\frac{2}{n+2}}, \ee
where $C$ is a positive constant which only depends on $n,\O, q$ and remains bounded as $q\rightarrow 1$. As a result, no matter $|\Gamma_1|\rightarrow 0$, $M_0\rightarrow 0$ or $q\rightarrow 1$, we will have $T^{*}\rightarrow\infty$.
\end{theorem}

Many work have been devoted to study the parabolic equation with Neumann boundary conditions which are analogous to (\ref{Prob, Simple Model}) but with $\Gamma_1=\p\O$. More precisely, they study the problem
\begin{equation}\label{Classical prob}
\left\{\begin{array}{lll}
u_{t}(x,t)-\Delta u(x,t)=f(x,t) &\text{in}& \Omega\times (0,T],\\
\frac{\partial u}{\partial n}(x,t)=F\big(x,t,u(x,t)\big) &\text{on}& \p\O\times (0,T],\\
u(x,0)=\psi(x) &\text{in}& \Omega,
\end{array}\right.
\end{equation}
where $f\in C^{\alpha,\alpha/2}\big(\ol{\O}\times[0,T]\big)$, $F\in C\big(\p\O\times[0,\infty)\times(-\infty,\infty)\big)$ and $\psi\in C^{1}(\ol{\O})$. For example, \cite{Amann, ACR-B, Friedman, L-GMW} discussed the existence and uniqueness of the solution to (\ref{Classical prob}) by various methods and in different spaces. \cite{Fila, HY, LP, L-GMW, RR, Walter} studied the finite time blow-up of the solution and the upper bound of the blow-up time. \cite{PPV-P, PS1, PS2} estimated the lower bound of the blow-up time. \cite{FQ, Friedman, HY, L-GMW, RR} covered some other topics such as the localization of the blow-up points, the blow-up rate, the asymptotic behaviour near the blow-up points and so on. \cite{DL, Hu, Levine, QS} are books or surveys which summarized the work and methods about different issues on the problem (\ref{Classical prob}).

However, there have not been many works on the problem (\ref{Prob, Simple Model}) since the normal derivative $\frac{\p u}{\p n}$ in (\ref{Prob, Simple Model}) is not continuous along the boundary. However, some ideas from previous work can be borrowed to apply to (\ref{Prob, Simple Model}). 

For the theories on existence and uniqueness of the solution to (\ref{Prob, Simple Model}), one of the key tools is the Jump Relation of the Single-layer Potentials mentioned in (\cite{Friedman}, Sec. 2, Chap. 5). We discuss several variants of this Jump Relation in Appendix \ref{Sec, Jump Relation} such that they can be adapted to our problem. Then we follow the arguments in \cite{Friedman, L-GMW} to prove in Appendix \ref{Sec, exist. and uniq.} the existence and uniqueness of the solution to the Linear problem (\ref{Linear Prob, Simple Model}) and to the Nonlinear problem (\ref{Nonlinear Prob, Simple Model}) by using the theories developed in Appendix \ref{Sec, Jump Relation}. Both proofs in Appendix \ref{Sec, Jump Relation} and Appendix \ref{Sec, exist. and uniq.} are very tedious and analogous to previous work, so we decide to put them into the Appendices.

To estimate the upper bound of the blow-up time, there have been existing several methods. For our problem (\ref{Prob, Simple Model}), the simplest one to apply seems to be \cite{RR}, in which it introduces a suitable energy function and shows the finite blow-up of this energy function. The process is very succinct and even gives an explicit formula for the upper bound. We follow this idea in \cite{RR} but utilize a sequence of approximated solutions $\{v_j\}_{j\geq 1}$, which satisfy the approximated problem (\ref{smoother eq. by cut-off fcn}), to justify all the calculations.

The lower bound of the blow-up time is usually harder to obtain, a popular method dealing with the lower bound is established in \cite{PS1, PS2}. After that, the similar idea is also applied to some more generalized problems, see e.g. \cite{BS, LL, PPV-P}. This method also introduces a suitable energy function and derives a differential inequality for that energy function, from there the lower bound can be achieved. However, when deriving the differential inequality, their technique requires the convexity of $\O$ and is not applicable to the partial boundary problem, e.g. (\ref{Prob, Simple Model}) with $\Gamma_2\neq\emptyset$. In addition, their arguments only consider $n=2$ or $3$. Thus, in order to handle (\ref{Prob, Simple Model}) without any of these limitations, we seek a different way by directly analyzing the Representation formula (\ref{Representation formula for bdry.}) of $u^{*}$ and taking advantage of the properties of the heat kernel. In this way, we are able to give a lower bound of the blow-up time as in Theorem \ref{Thm, lower bound of blow-up time}. 

The organization of this paper is as following: Section \ref{Sec, upper bound} is devoted to show Theorem \ref{Thm, upper bound of blow-up time}. Then we prove Theorem \ref{Thm, lower bound of blow-up time} in Subsection \ref{Secsub, lower bound} by analyzing the Representation formula (\ref{Representation formula for bdry.}) which is derived in Subsection \ref{Secsub, Weak Soln and Rep. Form.}. Subsection \ref{Secsub, Compare lower bounds} compares the lower bound estimate derived from our method with previous results. Section \ref{Sec, numerical result} presents some numerical simulations. Appendix \ref{Sec, Jump Relation} introduces some generalized Jump Relations of the Single-layer Potentials. Appendix \ref{Sec, exist. and uniq.} establishes the general theories on the existence and uniqueness of the solution and verifies Theorem \ref{Thm, fundamental theorem} as a special case.

\section{Upper Bound of the Blow-up Time}
\label{Sec, upper bound}
First of all, we want to point out that Theorem \ref{Thm, fundamental theorem} has been verified (See Remark \ref{Remark, app. to exist. and uniq.}). Then based on this fundamental result, the goal of this section is to prove the unique solution $u^{*}$ of (\ref{Prob, Simple Model}) always blows up (i.e. $L^{\infty}$ norm of $u^{*}$ goes to $\infty$) in finite time. In addition, the blow-up time $T^{*}$ is estimated in terms of $|\Gamma_1|$ and the initial data $u_0$, as long as $u_0$ is positive on $\ol{\O}$.

A common way to prove the blow-up of a solution is to introduce a suitable energy function related to that solution and then derive a differential inequality to show the energy function blows up. This process usually involves integration by parts and therefore requires $Du$ (i.e. the derivative with respect to the space variable) to be continuous up to the boundary. However, $u^{*}$ is not such smooth, since the normal derivative $\frac{\p u}{\p n}$ is apparently not continuous along $\wt{\Gamma}$. Thus, some approximations are needed to get through this process.

Firstly, we approximate the domain $\O$ from inside by $\{\O_k\}_{k\geq 1}$ which are defined as: for any $k\geq 1$,  
\be\label{def of approxiamted domain}
\O_{k}\triangleq \{x\in\O: \text{dist}(x,\p\O)>1/k\}. \ee
In addition, for any $x\in\p\O_{k}$, we use $\ora{n_k}(x)$ to denote the exterior unit normal vector at $x$ with respect to $\p\O_{k}$ while for any $x\in\p\O$, $\ora{n}(x)$ represents the exterior unit normal vector at $x$ with respect to $\p\O$. 

Secondly, we approximate $u^{*}$ by introducing a sequence of cut-off functions $\{\eta_{j}\}_{j\geq 1}$. More specifically, we choose a sequence of boundary pieces $\{\Gamma_{1,j}\}_{j\geq 1}$ such that $\Gamma_{1,j}\subset\Gamma_1$ and $\Gamma_{1,j}\nearrow\Gamma_1$, (see Figure \ref{Fig, boundary approximation}). Then we define a sequence of $C^{\infty}$ cut-off functions $\{\eta_j\}_{j=1}^{\infty}$ such that for each $j\geq 1$, 
\be\label{cut-off fcns}
\eta_{j}(x)\;\left\{\begin{array}{ll}
=1, & \quad x\in\Gamma_{1,j},\\
\in[0,1], & \quad x\in\ol{\Gamma}_1\setminus\Gamma_{1,j},\\
=0, & \quad x\in\p\O\setminus\Gamma_1.
\end{array}\right.\ee
In addition, we require that $\eta_{j+1}(x)\geq \eta_{j}(x)$, for any $j\geq 1$ and for any $x\in\p\O$.

\begin{figure}[!ht] \centering
\begin{tikzpicture}[scale=0.8]
\begin{large}
\draw [domain=0:180] plot({4*cos(\x)},{4*sin(\x)});
\path (-4,0) coordinate(A);
\path (0,4) coordinate(B);
\draw (A) node [right] {$\Gamma_1$};
\draw (B) node [color=red] [below] {$\Gamma_{1,j}$};
\draw [color=red] [domain=260:310] plot({3.7+0.5*cos(\x)},{1.58+0.5*sin(\x)});
\draw [color=red] [domain=230:280] plot({-3.7+0.5*cos(\x)},{1.58+0.5*sin(\x)});
\end{large}
\end{tikzpicture}
\caption{$\Gamma_{1,j}$}
\label{Fig, boundary approximation} 
\end{figure}
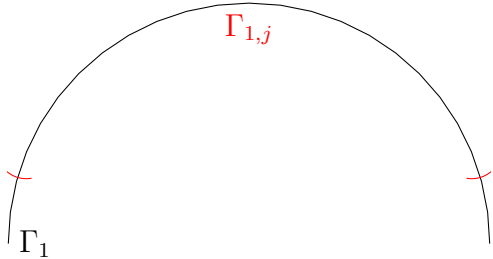 

Now for each $j\geq 1$, one considers the following problem 
\begin{equation}\label{smoother eq. by cut-off fcn}
\left\{\begin{array}{lll}
(v_j)_{t}(x,t)=\Delta v_{j}(x,t) &\text{in}& \Omega\times (0,T],\\
\frac{\partial v_j}{\partial n}(x,t)=\eta_j(x)\,{v_j}^{q}(x,t) &\text{on}& \p\O\times (0,T],\\
v_j(x,0)=u_0(x) &\text{in}& \Omega.
\end{array}\right.
\end{equation}
If we take $f\equiv 0$, $\eta\equiv\eta_j$, $F(\lam)=\lam^{q}$ for $\lam\in\m{R}$, $\psi\equiv u_0$ and $\Gamma_2=\emptyset$ in (\ref{Nonlinear Prob, Simple Model}), then by Appendix \ref{Subsec, Nonlinear Case}, we know (\ref{smoother eq. by cut-off fcn}) has a nonnegative unique maximal solution $v_{j}\in C^{2,1}(\O\times\big(0,T_{j}^{*})\big)\cap C\big(\ol{\O}\times[0,T_{j}^{*})\big)$, where $T_{j}^{*}$ denotes the maximal existence time for (\ref{smoother eq. by cut-off fcn}). In addition, by Corollary \ref{Cor, comparison for cut-off fcn.} and (\ref{cut-off fcns}), $v_j\leq u^{*}$ on $\ol{\O}\times[0,T]$ for any $T<\min\{T^{*},T_j^{*}\}$. Thus, using Theorem \ref{Thm, infinity norm characterizes the blow-up}, one has $T_j^{*}\geq T^{*}$ and $v_j\leq u^{*}$ on $\ol{\O}\times[0,T^{*})$.

\begin{lemma}\label{Lemma, convergence of flux}
Given $0<t_1<t_2<\infty$ and suppose that $\phi:\ol{\O}\times[t_1,t_2]\rightarrow\m{R}^{n}$ is continuous, then 
\[\lim_{k\rightarrow\infty}\int_{t_1}^{t_2}\int_{\p\O_{k}}\phi(x,t)\cdot\ora{n_{k}}(x)\,dS_{x}\,dt=\int_{t_1}^{t_2}\int_{\p\O}\phi(x,t)\cdot\ora{n}(x)\,dS_{x}\,dt.\]
\end{lemma}

\begin{proof}
Since $\p\O\in C^{2}$, we know it has the interior ball property. Therefore when $k$ is large enough, the mapping $\Psi_k: \p\O\rightarrow\p\O_{k}$, which is defined as 
\[\Psi_k(\xi)=\xi-\frac{1}{k}\,\ora{n}(\xi),\quad\forall\, \xi\in\p\O,\]
is a bijection. Moreover, one can see that $\Psi_k$ is $C^1$ and $\ora{n_k}\circ\Psi_k$ is continuous on $\p\O$. As a result,
\begin{align*}
& \int_{\p\O_{k}}\phi(x,t)\cdot\ora{n_k}(x)\,dS_x \\
\xlongequal[x=\Psi_{k}(\xi)]{\xi=\Psi_{k}^{-1}(x)}& \int_{\p\O}\phi\big(\Psi_{k}(\xi),t\big)\cdot\ora{n_k}\big(\Psi_{k}(\xi)\big)\big|J\Psi_{k}(\xi)\big|\,dS_{\xi}.
\end{align*}
Then by Lebesgue's dominated convergence theorem, 
\[\lim_{k\rightarrow\infty}\int_{\p\O_{k}}\phi(x,t)\cdot\ora{n_k}(x)\,dS_x=\int_{\p\O}\phi(\xi,t)\cdot\ora{n}(\xi)\,dS_{\xi}.\]
In addition, we know that $\int_{\p\O_{k}}\phi(x,t)\cdot\ora{n_k}(x)\,dS_x$ is uniformly bounded in $k$ and $t$, since $\phi$ is bounded. It then again follows from 
Lebesgue's dominated convergence theorem that
\[\lim_{k\rightarrow\infty}\int_{t_1}^{t_2}\int_{\p\O_{k}}\phi(x,t)\cdot\ora{n_{k}}(x)\,dS_{x}\,dt=\int_{t_1}^{t_2}\int_{\p\O}\phi(\xi,t)\cdot\ora{n}(\xi)\,dS_{\xi}\,dt.\]

\end{proof}

Now we start to prove Theorem \ref{Thm, upper bound of blow-up time}.

\begin{proof} 
Firstly by Theorem \ref{Thm, fundamental theorem}, $u^{*}(x,t)>0$ for any $x\in\ol{\O}$ and $t>0$, so to judge whether the solution blows up or not, we can assume \[\inf_{x\in\O}u_0(x)\geq \v_{0}\] for some positive constant $\v_{0}$. Otherwise, we start from any positive time.

Secondly, for any $T\in(0,T^{*})$, we fix it temporarily and denote \[M_T\triangleq ||u^{*}||_{L^{\infty}(\ol{\O}\times[0,T])}.\]
For any $j\geq 1$, recalling that $v_j$ is the solution to (\ref{smoother eq. by cut-off fcn}), then by the maximum principle and the fact that $v_j\leq u^{*}$, we have
\be\label{bound for approximated sol'n}
\v_0\leq v_j(x,t)\leq M_T, \quad \forall\, (x,t)\in \overline{\O}\times[0,T].
\ee
Moreover, it follows from (\cite{Friedman}, the last Corollary, Sec. 4, Chap. 5) that for any $\tau_0>0$ and $1\leq i\leq n$, 
\begin{equation}\label{cont. deri.}
(v_j)_{x_i}\in C(\overline{\O}\times[\tau_0,T]).
\end{equation}

Borrowing an idea from \cite{RR}, for any $j\geq 1$ and $k\geq 1$, we define $h_{j,k}:[0,T]\rightarrow\m{R}$ and 
$h_j:[0,T]\rightarrow\m{R}$ by
\[h_{j,k}(t)=\int_{\O_{k}}v_{j}^{1-q}(x,t)\,dx\] and
\[h_{j}(t)=\int_{\O}v_{j}^{1-q}(x,t)\,dx.\]
Since $v_j\in C^{2,1}\big(\O\times[0,T^{*})\big)$, the following calculations are justified.

\begin{eqnarray*}
h_{j,k}'(t) &=& (1-q)\int_{\O_{k}}v_{j}^{-q}\,(v_j)_{t}\,dx\\
&=& (1-q)\int_{\O_{k}}v_{j}^{-q}\,\Delta v_j\,dx\\
&=& (1-q)\int_{\O_{k}}\,\nabla\cdot(v_j^{-q}Dv_j)+q\,v_{j}^{-q-1}|Dv_j|^2\,dx\\
&\leq & (1-q)\int_{\O_{k}}\nabla\cdot(v_j^{-q}Dv_j)\,dx\\
&=& (1-q)\int_{\p \O_{k}}v_j^{-q}\,Dv_j\cdot \ora{n_{k}}\,dS.
\end{eqnarray*}
Integrating $t$ from $\tau_0$ to $T$,
\[h_{j,k}(T)-h_{j,k}(\tau_0)\leq (1-q)\int_{\tau_0}^{T}\int_{\p\O_{k}}v_j^{-q}(x,t)\,D v_j(x,t)\cdot\ora{n_{k}}(x)\,dS_{x}\,dt.\]
Taking $k\rightarrow\infty$, by (\ref{cont. deri.}) and Lemma \ref{Lemma, convergence of flux} with $\phi=v_j^{-q}\,Dv_j$, we attain
\begin{align*}
h_j(T)-h_j(\tau_0) &\leq (1-q)\int_{\tau_0}^{T}\int_{\p\O}v_j^{-q}(x,t)\, D v_j(x,t)\cdot\ora{n}(x)\,dS_{x}\,dt \\
&= (1-q)\int_{\tau_0}^{T}\int_{\p\O}\eta_{j}(x)\,dS_{x}\,dt\\
&\leq (1-q)\int_{\tau_0}^{T}|\Gamma_{1,j}|\,dt\\
&= (1-q)(T-\tau_0)|\Gamma_{1,j}|.
\end{align*}
Sending $\tau_0\rightarrow 0$ and noticing (\ref{bound for approximated sol'n}), we obtain $h_j(T)-h_j(0)\leq (1-q)\,T\,|\Gamma_{1,j}|$, i.e.
\[\int_{\O}v_j^{1-q}(x,T)\,dx\leq \int_{\O}u_0^{1-q}(x)\,dx+(1-q)\,T\,|\Gamma_{1,j}|.\]
Consequently,
\[0\leq \int_{\O}u_0^{1-q}(x)\,dx+(1-q)\,T\,|\Gamma_{1,j}|.\]
Then due to the fact that $|\Gamma_{1,j}|\rightarrow |\Gamma_1|$,
\[T\leq \frac{1}{(q-1)|\Gamma_1|}\int_{\O}u_0^{1-q}(x)\,dx.\]
Finally, since $T$ is arbitrary in $(0,T^{*})$, then 
\[T^{*}\leq \frac{1}{(q-1)|\Gamma_1|}\int_{\O}u_0^{1-q}(x)\,dx.\]
Thus, we have shown the solution must blow up in finite time and derived an upper bound for $T^{*}$. Now it follows from Theorem \ref{Thm, infinity norm characterizes the blow-up} that 
\[\sup_{(x,t)\in\ol{\O}\times[0,T^{*})}|u^{*}(x,t)|=\infty.\]
Consequently, Theorem \ref{Thm, upper bound of blow-up time} is verified.
\end{proof}

\section{Lower Bound of the Blow-up Time}
\label{Sec, lower bound}

\subsection{Derivation of  the lower bound}
\label{Secsub, lower bound}
In this subsection, we will derive a lower bound for $T^{*}$ by analyzing the Representation formulas (\ref{Representation formula for bdry.}). Here we want to point out that (\ref{Representation formula for bdry.}) is a formula for $u^{*}$ on $\p\O\times[0,T^{*})$ which is derived from (\ref{Representation formula for inside}), the formula for $u^{*}$ on $\O\times[0,T^{*})$. To estimate $T^{*}$, due to the maximum principle, it suffices to study the boundary values, thus we just analyze (\ref{Representation formula for bdry.}). 

The proof of (\ref{Representation formula for bdry.}) requires some work, but the idea is not hard. Firstly, if the solution $u^{*}$ is smooth up to the boundary, then using integration by parts, one can easily see that $u^{*}$ is also a weak solution as in Definition \ref{Def, weak soln}. From there, (\ref{Representation formula for bdry.}) can be obtained by some standard steps given in Theorem \ref{Thm, weak sol'n implies rep. formula} and its Corollary \ref{Cor, rep. formula for bdry.}. However, the solution $u^{*}$ is not $C^{1}$ up to the boundary, so again we need to take an approximation procedure, which is similar to that in Section \ref{Sec, upper bound}. This process is tedious, so we postpone it to Subsection \ref{Secsub, Weak Soln and Rep. Form.}.  

Before the proof of Theorem \ref{Thm, lower bound of blow-up time}, let's state the following Lemma \ref{Lemma, bound for surface integral} which will be used for several times in this subsection.

\begin{lemma}\label{Lemma, bound for surface integral}
Suppose $\O$ is an open, bounded set in $\m{R}^{n}$ with $\p\O\in C^{2}$, then there exists a constant $C=C(n,\O)$ such that for any $x\in\p\O$ and $t>0$,
\[\frac{1}{t^{(n-1)/2}}\int_{\p\O}e^{-\frac{|x-y|^2}{4t}}\,dS_{y}\leq C.\]
\end{lemma}
\begin{proof}
It is readily to show this conclusion by taking advantage of the definition of a $C^{2}$ boundary, we omit the proof here.
\end{proof}

Now we start to prove Theorem \ref{Thm, lower bound of blow-up time}. 
\begin{proof}
In the following, $C$ will be used to denote a positive constant which only depends on $n$, $\O$, $q$ and is bounded when $q\rightarrow 1$. Moreover, $C$ may be different from line to line. We prove by analyzing the following Representation formula (\ref{Representation formula for bdry.}) for $u^{*}$ on the boundary points $(x,t)\in\p\O\times[0,T^{*})$:
\begin{align}\label{bdry. point estimate decomp.}
u^{*}(x,t) =& 2\int_{\O}\Phi(x-y,t)\,u_{0}(y)\,dy+ 2\int_{0}^{t}\int_{\Gamma_1}\Phi(x-y, t-\tau)\,(u^{*})^{q}(y,\tau)\,dS_{y}\,d\tau \nonumber \\
& +2\int_{0}^{t}\int_{\p\O}(D\Phi)(x-y,t-\tau)\cdot\ora{n}(y)\,u^{*}(y,\tau)\,dS_{y}\,d\tau \nonumber \\
\triangleq &  I+II+III.
\end{align}

Define $M,\, \widetilde{M}:[0,T^{*})\rightarrow\m{R}$ by \[M(t)=\max_{y\in\p\O}u^{*}(y,t)\] and
\[\widetilde{M}(t)=\max\limits_{\tau\in[0,t]}M(\tau).\]
It is clear that $\widetilde{M}$ blows up at the same time $T^{*}$ as $u^{*}$. It is also easy to see  
\begin{equation}\label{I ineq. by max. initial}
I\leq 2M_0,
\end{equation}
and
\be\label{II ineq. by double integral}\begin{split}
II &\leq C\int_{0}^{t}M^{q}(\tau)\int_{\Gamma_1}(t-\tau)^{-n/2}\,e^{-\frac{|x-y |^{2}}{4(t-\tau)}}\,dS_{y}\,d\tau  \\
&= C\int_{0}^{t}M^{q}(t-\tau)\int_{\Gamma_1}\tau^{-n/2}\,e^{-\frac{|x-y |^{2}}{4\tau}}\,dS_{y}\,d\tau.
\end{split}\ee
In addition, we have that
\begin{align}\label{III ineq. by double integral}
III &\leq C\int_{0}^{t}M(\tau)\int_{\p\O}\frac{|x-y|^2}{(t-\tau)^{n/2+1}}\,e^{-\frac{|x-y|^2}{4(t-\tau)}}\,dS_{y}\,d\tau \nonumber \\
&= C\int_{0}^{t}M(t-\tau)\int_{\p\O}\frac{|x-y|^2}{\tau^{n/2+1}}\,e^{-\frac{|x-y|^2}{4\tau}}\,dS_{y}\,d\tau \nonumber \\
&\leq C\int_{0}^{t}M(t-\tau)\int_{\p\O}\frac{1}{\tau^{n/2}}\,e^{-\frac{|x-y|^2}{8\tau}}\,dS_{y}\,d\tau \nonumber \\
&\leq C\int_{0}^{t}M(t-\tau)\,\tau^{-1/2}\,d\tau,
\end{align}
where the first inequality  is due to Lemma \ref{Lemma, geom. prop. of C^2 bdry}, the last inequality is due to Lemma \ref{Lemma, bound for surface integral} and the second inequality is because
\[\frac{|x-y|^2}{\tau}e^{-\frac{|x-y|^2}{8\tau}}\leq\sup_{r\geq 0}\,r\,e^{-r/8}\leq C.\]

Now we are trying to estimate (\ref{II ineq. by double integral}) and (\ref{III ineq. by double integral}) further. Taking $m=1+\frac{1}{n+1}$, it follows from Holder's inequality that 
{\large \begin{align*}
\int_{\Gamma_1}\tau^{-n/2}e^{-\frac{|x-y|^2}{4\tau}}\,dS_{y} &\leq \tau^{-n/2}\bigg(\int_{\Gamma_1}e^{-\frac{m|x-y|^2}{4\tau}}\,dS_{y}\bigg)^{1/m}|\Gamma_1|^{(m-1)/m} \\
&= \tau^{-\frac{n}{2}+\frac{n-1}{2m}} \bigg(\int_{\Gamma_1}\tau^{-\frac{n-1}{2}}e^{-\frac{m|x-y|^2}{4\tau}}\,dS_{y}\bigg)^{1/m}|\Gamma_1|^{(m-1)/m} \\
&\leq C\,\tau^{-\frac{n}{2}+\frac{n-1}{2m}}\,|\Gamma_1|^{\frac{m-1}{m}} \\
&= C\,\tau^{-\frac{2n+1}{2n+4}}|\Gamma_1|^{\frac{1}{n+2}},
\end{align*}}
where the second inequality is because of Lemma \ref{Lemma, bound for surface integral}. Thus (\ref{II ineq. by double integral}) leads to
\be\label{II ineq. by single integral}
II \leq C\,|\Gamma_1|^{\frac{1}{n+2}}\int_{0}^{t}M^{q}(t-\tau)\,\tau^{-\frac{2n+1}{2n+4}}\,d\tau.
\ee
By Holder's inequality again, 
\begin{align*}
\int_{0}^{t}M^{q}(t-\tau)\,\tau^{-\frac{2n+1}{2n+4}}\,d\tau
&\leq \bigg(\int_{0}^{t}M^{\frac{qm}{m-1}}(t-\tau)\,d\tau\bigg)^{\frac{m-1}{m}}\bigg(\int_{0}^{t}\tau^{-\frac{(2n+1)m}{2n+4}}\,d\tau\bigg)^{\frac{1}{m}} \\
&= \bigg(\int_{0}^{t}M^{q(n+2)}(\tau)\,d\tau\bigg)^{\frac{1}{n+2}}\bigg(\int_{0}^{t}\tau^{-\frac{2n+1}{2n+2}}\,d\tau\bigg)^{\frac{n+1}{n+2}}\\
&= C\,t^{\frac{1}{2n+4}}\,\bigg(\int_{0}^{t}M^{q(n+2)}(\tau)\,d\tau\bigg)^{\frac{1}{n+2}}.
\end{align*}
Based on this, (\ref{II ineq. by single integral}) becomes
\be\label{II's estimate}
II\leq C\,|\Gamma_1|^{\frac{1}{n+2}}\,t^{\frac{1}{2n+4}}\bigg(\int_{0}^{t}M^{q(n+2)}(\tau)\,d\tau\bigg)^{\frac{1}{n+2}}.
\ee 
To estimate $III$, it follows from (\ref{III ineq. by double integral}) and Holder's inequality that
\begin{align}\label{III's estimate}
III&\leq  C\,\bigg(\int_{0}^{t}M^{n+2}(\tau)\,d\tau\bigg)^{\frac{1}{n+2}}\bigg(\int_{0}^{t}\tau^{-\frac{1}{2}\frac{n+2}{n+1}}\,d\tau\bigg)^{\frac{n+1}{n+2}} \nonumber\\
&= C\,t^{\frac{n}{2n+4}}\bigg(\int_{0}^{t}M^{n+2}(\tau)\,d\tau\bigg)^{\frac{1}{n+2}}.
\end{align}

Combining (\ref{bdry. point estimate decomp.}), (\ref{I ineq. by max. initial}), (\ref{II's estimate}), (\ref{III's estimate}), we obtain
\begin{align*}
u^{*}(x,t) \leq & C\bigg[M_0+|\Gamma_1|^{\frac{1}{n+2}}\,t^{\frac{1}{2n+4}}\bigg(\int_{0}^{t}M^{q(n+2)}(\tau)\,d\tau\bigg)^{\frac{1}{n+2}}+ t^{\frac{n}{2n+4}}\bigg(\int_{0}^{t}M^{n+2}(\tau)\,d\tau\bigg)^{\frac{1}{n+2}}\bigg].
\end{align*}
Since $x$ is arbitrary on $\p\O$, by raising both sides to the power $n+2$,
\begin{align*}
M^{n+2}(t) &\leq C\bigg[M_0^{n+2}+|\Gamma_1|\,t^{1/2}\int_{0}^{t}M^{q(n+2)}(\tau)\,d\tau+t^{n/2}\int_{0}^{t}M^{n+2}(\tau)\,d\tau\bigg].
\end{align*}
As a consequence, 
\begin{align}\label{ineq. for maximal sol'n}
\widetilde{M}^{n+2}(t) &\leq C\bigg[M_0^{n+2}+|\Gamma_1|\,t^{1/2}\int_{0}^{t}M^{q(n+2)}(\tau)\,d\tau+t^{n/2}\int_{0}^{t}M^{n+2}(\tau)\,d\tau\bigg] \nonumber\\
&\leq C\bigg[M_0^{n+2}+|\Gamma_1|\,t^{1/2}\int_{0}^{t}\widetilde{M}^{q(n+2)}(\tau)\,d\tau+t^{n/2}\int_{0}^{t}\widetilde{M}^{n+2}(\tau)\,d\tau\bigg] \nonumber\\
&\leq C\big(1+t^{n/2}\big)\bigg[M_0^{n+2}+|\Gamma_1|\int_{0}^{t}\widetilde{M}^{q(n+2)}(\tau)\,d\tau+\int_{0}^{t}\widetilde{M}^{n+2}(\tau)\,d\tau\bigg].
\end{align}
We define
\be\label{energy fcn.}
E(t)=M_0^{n+2}+|\Gamma_1|\int_{0}^{t}\widetilde{M}^{q(n+2)}(\tau)\,d\tau+\int_{0}^{t}\widetilde{M}^{n+2}(\tau)\,d\tau,
\ee
then $E(0)=M_0^{n+2}$ and $E(t)$ is increasing. Moreover $E(t)$ also blows up at $T^{*}$, since $\widetilde{M}$ is increasing. Now (\ref{ineq. for maximal sol'n}) becomes
\[\widetilde{M}^{n+2}(t)\leq C\big(1+t^{n/2}\big)E(t)\]
and consequently 
\begin{align}\label{ODE for energy fcn}
E'(t) &=|\Gamma_1|\,\widetilde{M}^{q(n+2)}(t)+\widetilde{M}^{(n+2)}(t) \nonumber\\
&\leq C\,|\Gamma_1|\,\big(1+t^{n/2}\big)^{q}E^{q}(t)+C\big(1+t^{n/2}\big)E(t).
\end{align} 
This looks like the Bernoulli equation, so we multiply both sides by $E^{-q}(t)$ and define $\Psi(t)\triangleq E^{1-q}(t)$, then $\Psi(t)\rightarrow 0$ as $t$ approaches to $T^{*}$ and 
\be\label{linear ode for reciprocal of the energy}
\Psi'(t)+C(q-1)\big(1+t^{n/2}\big)\Psi(t)\geq -C(q-1)|\Gamma_1|\,\big(1+t^{n/2}\big)^{q}.
\ee
We introduce the integration factor $\mu(t)$ which is defined as 
\[\mu(t)\triangleq\exp\bigg[C\int_{0}^{t}(q-1)\big(1+\tau^{n/2}\big)\,d\tau\bigg],\quad\forall\,t\geq 0.\]
It is easy to see that
\begin{align}\label{est. for int. factor}
\mu(t) &\leq \exp\Big[C\big(t+t^{1+\frac{n}{2}}\big)\Big] \nonumber\\
&\leq \exp\Big[C\big(1+t^{1+\frac{n}{2}}\big)\Big] \nonumber\\
&\leq C\,\exp\big(C\,t^{1+\frac{n}{2}}\big).
\end{align}
Multiplying (\ref{linear ode for reciprocal of the energy}) by $\mu(t)$, one gets
\[\big(\mu(t)\Psi(t)\big)'\geq -C(q-1)|\Gamma_1|\,\big(1+t^{n/2}\big)^{q}\mu(t).\]
Integrating this inequality and noticing that $\mu(0)\Psi(0)=M_0^{-(n+2)(q-1)}$, one obtains
\be\label{ineq. for the recip. energy fcn}
\mu(t)\Psi(t)\geq M_0^{-(n+2)(q-1)}-C(q-1)|\Gamma_1|\,\int_{0}^{t}\big(1+\tau^{n/2}\big)^{q}\,\mu(\tau)\,d\tau.\ee
It follows from (\ref{est. for int. factor}) that
\begin{align*}
\int_{0}^{t}\big(1+\tau^{n/2}\big)^{q}\,\mu(\tau)\,d\tau 
&\leq C\int_{0}^{t}\big(1+\tau^{n/2}\big)^{q}\,\exp\big(C\,\tau^{1+\frac{n}{2}}\big)\,d\tau\\
&\leq C\big(1+t^{n/2}\big)^{q}\,t\,\exp\big(C\,t^{1+\frac{n}{2}}\big)\\
&\leq C\exp\big[(C+1)\,t^{1+\frac{n}{2}}\big].
\end{align*}
Plugging in (\ref{ineq. for the recip. energy fcn}), we obtain 
\begin{equation*}
\mu(t)\Psi(t)\geq M_0^{-(n+2)(q-1)}-C(q-1)|\Gamma_1|\,\exp\big(C\,t^{1+\frac{n}{2}}\big).
\end{equation*}
Taking $t\rightarrow T^{*}$, one obtains
\begin{align*}
C(q-1)|\Gamma_1|\,\exp\Big[C\,(T^{*})^{1+\frac{n}{2}}\Big]
&\geq M_0^{-(n+2)(q-1)}\\
\exp\Big[C\,(T^{*})^{1+\frac{n}{2}}\Big] &\geq C^{-1}(q-1)^{-1}|\Gamma_1|^{-1}M_0^{-(n+2)(q-1)}\\
C\,(T^{*})^{\frac{n+2}{2}} &\geq \ln \Big(|\Gamma_1|^{-1}\Big)-(n+2)(q-1)\ln M_0-\ln(q-1)-\ln C.\\
\end{align*}
Hence, 
\bes T^{*}\geq C^{-\frac{2}{n+2}}\bigg[\ln \Big(|\Gamma_1|^{-1}\Big)-(n+2)(q-1)\ln M_0-\ln(q-1)-\ln C\bigg]^{\frac{2}{n+2}}. \ees
\end{proof}

\begin{remark}
From Theorem \ref{Thm, lower bound of blow-up time}, one can see that $T^{*}$ goes to infinity if $|\Gamma_1|\rightarrow0$, $M_0\rightarrow 0$ or $q\rightarrow 1$. In addition, when $|\Gamma_1|\rightarrow 0$, the lower bound of $T^{*}$ given in Theorem \ref{Thm, lower bound of blow-up time} is about the size
\[\Big[\ln \big(|\Gamma_1|^{-1}\big)\Big]^{2/(n+2)}.\]
\end{remark}

\subsection{Comparison to previous work}
\label{Secsub, Compare lower bounds}

Although many papers study the heat equation under Dirichlet or Neumann boundary conditions, few of them deal with the lower bound of the blow-up time. One of the influential work in this area is \cite{PS1}, however, their method has some limitations:
\begin{itemize}
\item[(i)] $n=2$ or $3$;
\item[(ii)] $\O$ needs to be convex;
\item[(iii)] Only applicable to the case when $\Gamma_1=\p\O$. 
\end{itemize} 
As we have seen in Subsection \ref{Secsub, lower bound}, our method of analyzing the Representation formula does not have any of these limitations. To show another advantage of this method, we compare with \cite{PS1} under the same limitations (i), (ii) and (iii). 

But this time we will not directly use the result (\ref{lower bound}) since that estimate focuses on the behavior of $T^{*}$ when $|\Gamma_1|\rightarrow 0$. As a result, although still applying Representation formula (\ref{Representation formula for bdry.}), we focus on $M_0\triangleq \max\limits_{\ol{\O}}u_0$ instead of $|\Gamma_1|$ to derive the lower bound (\ref{lower bound for the whole bdry.}). Then it is shown in the end of this subsection that if $M_0$ is big and $u_0$ does not change much on $\ol{\O}$, then our estimate (\ref{lower bound for the whole bdry.}) is larger than the result in \cite{PS1}.

In \cite{PS1}, when $n=3$, the authors consider the energy function
\[\phi(t)\triangleq \int_{\O}(u^{*})^{2m}(x,t)\,dx,\] 
which they prove to blow up at the same time $T^{*}$ as $u^{*}$ for $m\geq 2q-2$. They derive a first order differential inequality for $\phi(t)$ by using a technique developed in \cite{PS2} and show $\phi(t)$ remains bounded before some time $T_0$, therefore $T_0$ is a lower bound for $T^{*}$. From their paper, they attain the following estimate:
\be\label{lower bound in 2009, n=3}
T^{*}\geq C\,\bigg(\int_{\O}u_{0}^{2m}(x,t)\,dx\bigg)^{-2}.\ee 
When $n=2$, the arguments are similar and they obtain
\be\label{lower bound in 2009, n=2}
T^{*}\geq C\,\bigg(\int_{\O}u_{0}^{4m}(x,t)\,dx\bigg)^{-1}.\ee 

Now assuming (i)-(iii) and $M_0\geq 1$, we will estimate similarly as we did in Subsection \ref{Secsub, lower bound} to obtain a lower bound, but focusing on $M_0$ instead of $|\Gamma_1|$, since $\Gamma_1$ has been fixed to be $\p\O$ now. The following notations are as same as those in Subsection \ref{Secsub, lower bound}. Based on (\ref{energy fcn.}), (\ref{ODE for energy fcn}) and noticing the assumption $|\Gamma_1|=|\p\O|$, we have \[E(0)=M_0^{n+2}\] and 
\begin{align}\label{ODE for energy fcn for whole bdry.}
E'(t) &\leq C\Big[\big(1+t^{n/2}\big)^{q}E^{q}(t)+\big(1+t^{n/2}\big)E(t)\Big].
\end{align} 
Due to the assumption that $M_0\geq 1$, then $E(t)\geq 1$ for all $t\geq 0$.  
As a consequence, (\ref{ODE for energy fcn for whole bdry.}) implies
\be\label{standard ODE for energy fcn for whole bdry.}
E'(t)\leq C\big(1+t^{n/2}\big)^{q}E^{q}(t).
\ee
Integrating (\ref{standard ODE for energy fcn for whole bdry.}), we get
\begin{align*}
\int_{0}^{t}\frac{E'(\tau)}{E^{q}(\tau)}\,d\tau &\leq C\int_{0}^{t}\big(1+\tau^{n/2}\big)^{q}\,d\tau, \\
E^{1-q}(t)-M_0^{(1-q)(n+2)} &\geq (1-q)\,C\,\Big(t+t^{1+\frac{nq}{2}}\Big), \\
E^{1-q}(t) & \geq M_0^{-(q-1)(n+2)}-(q-1)\,C\,\Big(t+t^{1+\frac{nq}{2}}\Big).
\end{align*} 
Taking $t\rightarrow T^{*}$,
\bes (q-1)\,C\,\Big[T^{*}+(T^{*})^{1+\frac{nq}{2}}\Big]\geq M_0^{-(q-1)(n+2)}. \ees
If $T^{*}\leq 1$, then $T^{*}\geq (T^{*})^{1+\frac{nq}{2}}$ and therefore 
\[(q-1)\,C\,T^{*}\geq M_0^{-(q-1)(n+2)}.\]
Hence we obtain the following statement.
\begin{theorem}\label{Thm, lower bound for the whole bdry.}
Suppose $T^{*}$ is the maximal existence time for (\ref{Prob, Simple Model}) with $\Gamma_1=\p\O$ and $M_0\geq 1$, then  
\be\label{lower bound for the whole bdry.}
T^{*}\geq \min\Big\{1, M_0^{-(q-1)(n+2)}/C\Big\},\ee
where $C$ is a positive constant which only depends on $n$,  $\O$ and $q$. 
\end{theorem}

Now let's compare (\ref{lower bound in 2009, n=3}) and (\ref{lower bound in 2009, n=2}) with (\ref{lower bound for the whole bdry.}). If $M_0$ is very large and the initial function $u_0$ does not oscillate too much, then both (\ref{lower bound in 2009, n=3}) and (\ref{lower bound in 2009, n=2}) are about the size $M_0^{-4m}$. On the other hand, (\ref{lower bound for the whole bdry.}) is about the size $M_0^{-(q-1)(n+2)}$. Recalling that $m\geq 2q-2$, so we have
\[4m\geq 8(q-1)\geq (q-1)(n+2),\]
for no matter $n=2$ or $3$. Thus, our estimate for the lower bound is larger .

\subsection{Weak Solution and Representation Formula}
\label{Secsub, Weak Soln and Rep. Form.}
By Theorem \ref{Thm, fundamental theorem}, there exists a nonnegative unique maximal solution $u^{*}\in C^{2,1}\big(\O\times(0,T^{*})\big)\cap C\big(\overline{\O}\times [0,T^{*})\big)$ to (\ref{Prob, Simple Model}) such that
\begin{equation}\label{Prob maximal}
\left\{\begin{array}{lll}
u^{*}_{t}(x,t)=\Delta u^{*}(x,t) &\text{in}& \Omega\times (0,T^{*}),\\
\frac{\partial u^{*}}{\partial n}(x,t)=(u^{*})^{q}(x,t) &\text{on}& \Gamma_1\times (0,T^{*}),\\
\frac{\partial u^{*}}{\partial n}(x,t)=0 &\text{on}& \Gamma_2\times (0,T^{*}),\\
u^{*}(x,0)=u_0(x) &\text{in}& \Omega,
\end{array}\right.
\end{equation}
where $T^{*}$ is the maximal existence time as in Definition \ref{Def, maximal sol'n, Target Prob}. In this subsection, we will first verify that the solution $u^{*}$ to (\ref{Prob, Simple Model}) is also a weak solution (See Definition \ref{Def, weak soln}) and then derive Representation formulas (\ref{Representation formula for inside}) and (\ref{Representation formula for bdry.}) for $u^{*}$. 
\begin{definition}\label{Def, weak soln}
Suppose $T^{*}$ is the maximal existence time for (\ref{Prob, Simple Model}), then a function $u\in C\big(\ol{\O}\times[0,T^{*})\big)$ is called a weak solution of (\ref{Prob, Simple Model}) if for any $t\in(0,T^{*})$ and for any $\phi\in C^{2}\big(\overline{\O}\times[0,t]\big)$, 
\be\label{weak soln def}
\begin{split}
&\int_{0}^{t}\int_{\O}(\phi_{\tau}+\Delta\phi)(y,\tau)\,u^{*}(y,\tau)\,dy\,d\tau =  \int_{\O}\phi(y,t)\,u^*(y,t)-\phi(y,0)\,u_0(y)\,dy \\
& -\int_{0}^{t}\int_{\Gamma_1}\phi(y,\tau)\,(u^{*})^{q}(y,\tau)\,dS_{y}\,d\tau+\int_{0}^{t}\int_{\p\O}u^{*}(y,\tau)\,\frac{\p\phi}{\p n}(y,\tau)\,dS_{y}\,d\tau.
\end{split} \ee
\end{definition}

In order to prove $u^*$ satisfies (\ref{weak soln def}), we are again seeking some smoother approximations for $u^{*}$. Let $\{\eta_{j}\}_{j\geq 1}$ be the same sequence of cut-off functions as defined in (\ref{cut-off fcns}) and consider the following problem: for any $j\geq 1$,
\begin{equation}\label{approx. pb. with fixed nonlin. term}
\left\{\begin{array}{lll}
(w_j)_t(x,t)=\Delta w_j(x,t) &\text{in}& \O\times(0,T),\\
\frac{\p w_j}{\p n}(x,t)=\eta_{j}(x)\,(u^{*})^{q}(x,t) &\text{on}& \p\O\times(0,T),\\
w_j(x,0)=u_{0}(x) &\text{in}& \O.
\end{array}\right.
\end{equation}
Here we want to point out that (\ref{approx. pb. with fixed nonlin. term}) is different from (\ref{smoother eq. by cut-off fcn}), since (\ref{approx. pb. with fixed nonlin. term}) is linear in $w_j$ while (\ref{smoother eq. by cut-off fcn}) is nonlinear in $v_j$. As a result, the solution of (\ref{approx. pb. with fixed nonlin. term}) is easier to compare with $u^{*}$. 

Since $u^{*}$ is only defined in $\ol{\O}\times[0,T^{*})$, we should consider (\ref{approx. pb. with fixed nonlin. term}) only for $T\leq T^{*}$. In addition, by Subsection \ref{Subsec, Linear Case}, the maximal solution $w_j$ to (\ref{approx. pb. with fixed nonlin. term}) also exists until $T^{*}$, since (\ref{approx. pb. with fixed nonlin. term}) is linear in $w_j$. Thus, we conclude that $w_{j}\in C^{2,1}\big(\O\times(0,T^{*})\big)\cap C\big(\ol{\O}\times[0,T^{*})\big)$ . Moreover, it follows from (\cite{Friedman}, the last Corollary, Sec. 4, Chap. 5) that for any $1\leq i\leq n$ and $\tau_0>0$, 
\be\label{cont. deri. fixed}
\frac{\p w_{j}}{\p x_{i}}\in C\big(\overline{\O}\times[\tau_0 , T^{*})\big). \ee 
Hence, $w_j$ has better regularity than $u^{*}$ and this will help us to justify the calculations in the proof of Theorem \ref{Thm, max. sol'n implies weak sol'n}. But before that, let's demonstrate two other basic properties of $\{w_j\}_{j\geq 1}$, namely Lemma \ref{Lemma, bound for approx. sol'n} and Lemma \ref{Lemma, conv. of approx. sol'n}.

\begin{lemma}\label{Lemma, bound for approx. sol'n}
For any $j\geq 1$, $w_{j}\leq u^{*}$ on $\ol{\O}\times[0,T^{*})$.
\end{lemma}
\begin{proof}
Denoting $v_{j}=u^*-w_{j}$, then for any $T\in(0,T^{*})$, 
\begin{equation*}
\left\{\begin{array}{lll}
(v_{j})_{t}(x,t)-\Delta v_{j}(x,t)=0 &\text{in}& \O\times(0,T],\\
\frac{\p v_{j}}{\p n}(x,t)=\big[1-\eta_j(x)\big](u^{*})^{q}(x,t)\geq 0 &\text{on}& \p\O\times(0,T],\\
v_{j}(x,0)=0 &\text{in}& \O.
\end{array}\right.
\end{equation*}
By Corollary \ref{Cor, comparison and uniqueness for linear prob. of simple model},  $v_{j}\geq 0$ on $\ol{\O}\times[0,T]$. \end{proof}

The following pointwise convergence property is essential to justify the approximation process in the proof of Theorem \ref{Thm, max. sol'n implies weak sol'n}.

\begin{lemma}\label{Lemma, conv. of approx. sol'n}
For any $(x,t)\in\overline{\O}\times[0,T^{*})$, $\lim\limits_{j\rightarrow\infty} w_{j}(x,t)=u^*(x,t)$.
\end{lemma}
\begin{proof}
Since $w_j(x,0)=u_0(x)=u^{*}(x,0)$, it suffices to prove for any $T\in(0,T^{*})$, 
\[\lim_{j\rightarrow\infty} w_{j}(x,t)=u^*(x,t), \quad\forall\, (x,t)\in\ol{\O}\times(0,T].\]
We now fix $T\in(0,T^{*})$ and define $\eta^{*}:\Gamma_1\cup\Gamma_2\rightarrow\m{R}$ by
$$\eta^{*}(x)=\left\{\begin{array}{ll}
1, & \quad x\in\Gamma_1,\\
0, & \quad x\in\Gamma_2.
\end{array}\right.$$
Let $v_{j}=u^*-w_j$, then 
\be\label{pointwise conv. eq.}
\left\{\begin{array}{lll}
(v_j)_{t}(x,t)=\Delta v_j(x,t) & \quad\text{in}\quad & \O\times(0,T],\\
\frac{\p v_j}{\p n}(x,t)=[\eta^{*}(x)-\eta_{j}(x)]\,(u^{*})^q(x,t) & \quad\text{on}\quad & (\Gamma_1\cup\Gamma_2)\times(0,T],\\
v_j(x,0)=0 & \quad\text{in}\quad & \O.
\end{array}\right.\ee
Similar to the proof of Theorem \ref{Thm, exist. for linear simple model}, $v_j$ can be written in the following form 
\be\label{explicit expression for difference}
v_j(x,t)=\int_{0}^{t}\int_{\p\O}\Phi(x-y,t-\tau)\,\varphi_{j}(y,\tau)\,dS_{y}\,d\tau,\quad\forall\,(x,t)\in \overline{\O}\times(0,T], \ee
where $\varphi_j\in\mathcal{B}_{T}$ satisfies for any $(x,t)\in\big(\Gamma_1\cup\Gamma_2\big)\times(0,T]$,
\be\label{varphi eq.}
\varphi_j(x,t)=\int_{0}^{t}\int_{\p\O}K(x,t;y,\tau)\,\varphi_{j}(y,\tau)\,dS_{y}\,d\tau +H_{j}(x,t)\ee
with \[K(x,t;y,\tau)= -2(D\Phi)(x-y,t-\tau)\cdot\ora{n}(x)\] and 
\[H_{j}(x,t)=2[\eta^{*}(x)-\eta_{j}(x)]\,(u^{*})^q(x,t).\]
Since the function $K$ also satisfies (\ref{est. for the kernel of the induction def for bdry fcn}), we can follow the same way as the derivations of (\ref{est for accumulated kernel, linear}), (\ref{explicit formula with accumulated kernel, linear})  to obtain
\be\label{explicit formula for bdry fcn, point. conv., linear}
\varphi_j(x,t)=\int_{0}^{t}\int_{\p\O}K^{*}(x,t;y,\tau)\,H_j(y,\tau)\,dS_{y}\,d\tau + H_j(x,t) \ee
for some function $K^{*}$. Moreover, there exists a constant $C^{*}$, only depending on $n$, $\O$ and $T$, such that
\be\label{bound for accumulate kernel, point. conv., linear}
|K^{*}(x,t;y,\tau)|\leq C^{*}(t-\tau)^{-3/4}\,|x-y|^{-(n-3/2)}. \ee

Due to the choice of $\{\eta_j\}_{j\geq 1}$ and the fact that $u^{*}$ is bounded on $\overline{\O}\times[0,T]$, we know $H_{j}$ is also bounded on $\overline{\O}\times[0,T]$ and
\[\lim_{j\rightarrow\infty}H_{j}(x,t)=0,\quad\forall\, (x,t)\in\big(\Gamma_1\cup\Gamma_2\big)\times(0,T].\]
Then it follows from (\ref{explicit formula for bdry fcn, point. conv., linear}), (\ref{bound for accumulate kernel, point. conv., linear}) and the Lebesgue's dominated convergence theorem that 
\[\lim_{j\rightarrow\infty}\varphi_{j}(x,t)=0,\quad\forall\, (x,t)\in\big(\Gamma_1\cup\Gamma_2\big)\times(0,T].\]
In addition, the boundedness of $H_j$ implies the boundedness of $\varphi_{j}$, hence by (\ref{explicit expression for difference}) and the Lebesgue's dominated convergence theorem again, we get
\[\lim\limits_{j\rightarrow\infty}v_j(x,t)=0, \quad\forall\, (x,t)\in\overline{\O}\times(0,T].\] 
That is
\[\lim_{j\rightarrow\infty} w_{j}(x,t)=u^*(x,t), \quad\forall\, (x,t)\in\ol{\O}\times(0,T].\]
\end{proof}

Next, we will verify $u^*$ to be a weak solution by taking advantage of $\{w_j\}_{j\geq 1}$. 

\begin{theorem}\label{Thm, max. sol'n implies weak sol'n}
The maximal solution $u^{*}$ to (\ref{Prob, Simple Model}) is also a weak solution as in Definition \ref{Def, weak soln}.
\end{theorem}
\begin{proof}
Firstly, we choose the same sequence of domains $\{\O_{k}\}_{k\geq 1}$ as defined in (\ref{def of approxiamted domain}), i.e. $$\O_{k}=\{x\in\O: \text{dist}\,(x,\p\O)>1/k\}.$$ 
Then for any $0<\tau_0<t<T^{*},\,k\geq 1,\, j\geq 1$ and $\phi\in C^{2}\big(\overline{\O}\times[0,t]\big)$, we have
\[\int_{\tau_0}^{t}\int_{\O_k}(w_j)_t(y,\tau)\,\phi(y,\tau)\,dy\,d\tau=\int_{\tau_0}^{t}\int_{\O_k}\Delta w_j(y,\tau)\,\phi(y,\tau)\,dy\,d\tau. \]
Using integration by parts, 
\begin{align*}
&\int_{\tau_0}^{t}\int_{\O_{k}}(\phi_{\tau}+\Delta\phi)(y,\tau)w_j(y,\tau)\,dy\,d\tau = \int_{\O_{k}}\phi(y,t)w_j(y,t)-\phi(y,\tau_0)w_j(y,\tau_0)\,dy  \\
& -\int_{\tau_0}^{t}\int_{\p\O_{k}}\phi(y,\tau)\,D w_j(y,\tau)\cdot\ora{n_{k}}(y)\,dS_{y}\,d\tau+
\int_{\tau_0}^{t}\int_{\p\O_{k}}w_j(y,\tau)\,D\phi(y,\tau)\cdot \ora{n_{k}}(y)dS_{y}\,d\tau,
\end{align*}
where $\ora{n_{k}}$ denotes the exterior unit normal vector with respect to $\p\O_{k}$.

Sending $k\rightarrow\infty$, by (\ref{cont. deri. fixed}) and Lemma \ref{Lemma, convergence of flux}, we obtain 
\begin{align*}
&\int_{\tau_0}^{t}\int_{\O}(\phi_{\tau}+\Delta\phi)(y,\tau)\,w_j(y,\tau)\,dy\,d\tau = \int_{\O}\phi(y,t)\,w_j(y,t)-\phi(y,\tau_0)\,w_j(y,\tau_0)\,dy  \\
& -\int_{\tau_0}^{t}\int_{\p\O}\phi(y,\tau)\,D w_j(y,\tau)\cdot\ora{n}(y)\,dS_{y}\,d\tau+
\int_{\tau_0}^{t}\int_{\p\O}w_j(y,\tau)\,D\phi(y,\tau)\cdot \ora{n}(y)dS_{y}\,d\tau,
\end{align*}
Noticing $D w_j(y,\tau)\cdot\ora{n}(y)=\frac{\p w_j}{\p n}(y,\tau)=\eta_{j}(y)\,(u^{*})^{q}(y,\tau)$ on $\p\O$, we get
\begin{align*}
&\int_{\tau_0}^{t}\int_{\O}(\phi_{\tau}+\Delta\phi)(y,\tau)\,w_j(y,\tau)\,dy\,d\tau = \int_{\O}\phi(y,t)\,w_j(y,t)-\phi(y,\tau_0)\,w_j(y,\tau_0)\,dy  \\
& -\int_{\tau_0}^{t}\int_{\p\O}\phi(y,\tau)\,\eta_{j}(y)\,(u^{*})^{q}(y,\tau)\,dS_{y}\,d\tau+
\int_{\tau_0}^{t}\int_{\p\O}w_j(y,\tau)\,D\phi(y,\tau)\cdot \ora{n}(y)dS_{y}\,d\tau,
\end{align*}
Taking $j\rightarrow\infty$, then it follows from Lemma \ref{Lemma, bound for approx. sol'n}, Lemma \ref{Lemma, conv. of approx. sol'n} and the Lebesgue's dominated convergence theorem that
\bes\begin{split}
&\int_{\tau_0}^{t}\int_{\O}(\phi_{\tau}+\Delta\phi)(y,\tau)\,u^{*}(y,\tau)\,dy\,d\tau =  \int_{\O}\phi(y,t)\,u^*(y,t)-\phi(y,\tau_0)\,u^{*}(y,\tau_0)\,dy \\
& -\int_{\tau_0}^{t}\int_{\Gamma_1}\phi(y,\tau)\, (u^{*})^{q}(y,\tau)\,dS_{y}\,d\tau+\int_{\tau_0}^{t}\int_{\p\O}u^{*}(y,\tau)\,D\phi(y,\tau)\cdot \ora{n}(y)\,dS_{y}\,d\tau.
\end{split}\ees
Finally by sending $\tau_0\rightarrow 0$, we get (\ref{weak soln def}). 
\end{proof}

Next by (\ref{weak soln def}) and some standard steps, we are able to attain the Representation formula of $u^{*}$ for inside points.
\begin{theorem}\label{Thm, weak sol'n implies rep. formula}
For the maximal solution $u^{*}$ to (\ref{Prob, Simple Model}), it has the Representation formula for the inside points $(x,t)\in\O\times[0,T^{*})$,
\be\label{Representation formula for inside}
\begin{split}
u^{*}(x,t) = & \int_{\O}\Phi(x-y,t)\,u_0(y)\,dy+\int_{0}^{t}\int_{\Gamma_1}\Phi(x-y,t-\tau)\,(u^{*})^{q}(y,\tau)\,dS_{y}\,d\tau  \\
& +\int_{0}^{t}\int_{\p\O}(D\Phi)(x-y,t-\tau)\cdot \ora{n}(y)\,u^{*}(y,\tau)\,dS_{y}\,d\tau.
\end{split} \ee
\end{theorem}

\begin{proof}
For any $x\in\O$, $t\in(0,T^{*})$ and $\v>0$, we define $\phi^{x}:\ol{\O}\times[0,t)\rightarrow\m{R}$ by
\[\phi^{x}(y,\tau)=\Phi(x-y, t-\tau)=\frac{1}{(4\pi)^{n/2}}\frac{1}{(t-\tau)^{n/2}}e^{-\frac{|x-y|^2}{4(t-\tau)}}\]
and define $\phi^{x,\v}:\ol{\O}\times[0,t]\rightarrow\m{R}$ by
\[\phi^{x,\v}(y,\tau)=\Phi(x-y, t+\eps-\tau)=\frac{1}{(4\pi)^{n/2}}\frac{1}{(t+\eps-\tau)^{n/2}}e^{-\frac{|x-y|^2}{4(t+\eps-\tau)}}.\]
From these, one can see that $\phi^{x,\v}$ is smooth in its domain and 
$$(\phi^{x,\eps})_{\tau}(y,\tau)+\Delta_{y} (\phi^{x,\eps})(y,\tau)=0, \quad\forall\,(y,\tau)\in \ol{\O}\times[0,t].$$ Then we apply (\ref{weak soln def}) with $\phi=\phi^{x,\v}$ to attain
\begin{eqnarray*}
\int_{\Omega}\phi^{x,\eps}(y,t)\,u^{*}(y,t)\,dy &=&
\int_{\Omega}\phi^{x,\eps}(y,0)\,u_0(y)\,dy + \int_{0}^{t}\int_{\Gamma_1}\phi^{x,\eps}(y,\tau)\, (u^{*})^{q}(y,\tau)\,dS_{y}\,d\tau \\
&& -\int_{0}^{t}\int_{\p\Omega}\frac{\p \phi^{x,\eps}}{\p n}(y,\tau)\,u^{*}(y,\tau)\,dS_{y}\,d\tau.
\end{eqnarray*}
Sending $\v\rightarrow 0$, it follows from the Lebesgue's dominated convergence theorem that 
\begin{align*}
u^{*}(x,t) =& \int_{\O}\phi^{x}(y,0)\,u_0(y)\,dy+ \int_{0}^{t}\int_{\Gamma_1}\phi^{x}(y,\tau)\,(u^{*})^{q}(y,\tau)\,dS_{y}\,d\tau \\
&  -\int_{0}^{t}\int_{\p\O}\frac{\p\phi^{x}}{\p n}(y,\tau)\,u^{*}(y,\tau)\,dS_{y}\,d\tau \\
=& \int_{\O}\Phi(x-y,t)\,u_0(y)\,dy+\int_{0}^{t}\int_{\Gamma_1}\Phi(x-y,t-\tau)\,(u^{*})^{q}(y,\tau)\,dS_{y}\,d\tau  \\
& +\int_{0}^{t}\int_{\p\O}(D\Phi)(x-y,t-\tau)\cdot \ora{n}(y)\,u^{*}(y,\tau)\,dS_{y}\,d\tau.
\end{align*}
The last equality is because \[\frac{\p\phi^{x}}{\p n}(y,\tau)\triangleq D_{y}\big[\phi^{x}(y,\tau)\big]\cdot\ora{n}(y)=-(D\Phi)(x-y,t-\tau)\cdot \ora{n}(y).\]

Now we have proved (\ref{Representation formula for inside}) for $(x,t)\in\O\times(0,T^{*})$, next it is obvious to see that when $t\rightarrow 0$, both sides of (\ref{Representation formula for inside}) tend to $u_{0}(x)$. Thus (\ref{Representation formula for inside}) is also true for $(x,0)$, where $x\in\O$.
\end{proof}

Theorem \ref{Thm, weak sol'n implies rep. formula} only gives the formula for the inside points, but we still need the formula for the boundary points $(x,t)\in\p\O\times[0,T^{*})$ . In order to get that, we combine Theorem \ref{Thm, weak sol'n implies rep. formula} and Corollary \ref{Cor, Jump Relation for vary normal}. 

\begin{corollary}\label{Cor, rep. formula for bdry.}
For the maximal solution $u^{*}$ to (\ref{Prob, Simple Model}), it has the Representation formula for the boundary points $(x,t)\in\p\O\times[0,T^{*})$,
\be\label{Representation formula for bdry.}
\begin{split}
u^{*}(x,t) = & 2\int_{\O}\Phi(x-y,t)\,u_0(y)\,dy+2\int_{0}^{t}\int_{\Gamma_1}\Phi(x-y,t-\tau)\,(u^{*})^{q}(y,\tau)\,dS_{y}\,d\tau  \\
& +2\int_{0}^{t}\int_{\p\O}(D\Phi)(x-y,t-\tau)\cdot \ora{n}(y)\,u^{*}(y,\tau)\,dS_{y}\,d\tau.
\end{split} \ee
\end{corollary}
\begin{proof}
We fix any point $(x,t)\in \p\O\times(0,T^{*})$ and write $x_h=x-h\ora{n}(x)$ for $h>0$. As shown in the proof of Corollary \ref{Cor, Jump Relation for vary normal}, when $h$ is sufficiently small, $x_h\in \O$ for any $x\in\p\O$. Consequently we can apply Theorem \ref{Thm, weak sol'n implies rep. formula} to conclude that
\bes \begin{split}
u^{*}(x_h,t) = & \int_{\O}\Phi(x_h-y,t)\,u_0(y)\,dy+\int_{0}^{t}\int_{\Gamma_1}\Phi(x_h-y,t-\tau)\,(u^{*})^{q}(y,\tau)\,dS_{y}\,d\tau  \\
& +\int_{0}^{t}\int_{\p\O}(D\Phi)(x_h-y,t-\tau)\cdot \ora{n}(y)\,u^{*}(y,\tau)\,dS_{y}\,d\tau.
\end{split} \ees
Taking $h\rightarrow 0^{+}$, then it follows from Corollary \ref{Cor, Jump Relation for vary normal} that
\bes \begin{split}
u^{*}(x,t) = & \int_{\O}\Phi(x-y,t)\,u_0(y)\,dy+\int_{0}^{t}\int_{\Gamma_1}\Phi(x-y,t-\tau)\,(u^{*})^{q}(y,\tau)\,dS_{y}\,d\tau  \\
& +\int_{0}^{t}\int_{\p\O}(D\Phi)(x-y,t-\tau)\cdot \ora{n}(y)\,u^{*}(y,\tau)\,dS_{y}\,d\tau+\frac{1}{2}\,u^{*}(x,t),
\end{split} \ees
which implies (\ref{Representation formula for bdry.}). 

Now we have proved (\ref{Representation formula for bdry.}) for $(x,t)\in\p\O\times(0,T^{*})$. Next since we assume $\p\O\in C^{2}$, then for any $x\in\p\O$, 
\[\lim_{t\rightarrow 0^{+}}\int_{\O}\Phi(x-y,t)\,u_0(y)\,dy=\frac{1}{2}\,u_0(x).\]
As a result, when $t\rightarrow 0$, both sides of (\ref{Representation formula for bdry.}) tend to $u_{0}(x)$. Thus (\ref{Representation formula for bdry.}) is also true for $(x,0)$, where $x\in\p\O$.

\end{proof}

\section{Numerical Simulation}
\label{Sec, numerical result}
We have seen from the previous sections that as $|\Gamma_1|\rightarrow 0$, the upper bound of $T^{*}$ given in Theorem \ref{Thm, upper bound of blow-up time} is in the order $|\Gamma_1|^{-1}$ while the lower bound given in Theorem \ref{Thm, lower bound of blow-up time} is in the order $\big[\ln \big(|\Gamma_1|^{-1}\big)\big]^{2/(n+2)}$. The natural question is which order is more accurate? In this section, we try to numerically examine the order $\alpha$:
$T^{*}\sim C\, |\Gamma_1|^{-\alpha}.$

However, it has some difficulties to perform the simulation, since the blow-up of the solution destroys the accuracy of the numerical schemes. As a result, to ensure the accuracy, the numerical simulation should stop at some time $T_0$ before $u^{*}$ becomes large, say $\max\limits_{\ol{\O}\times[0,T_0]}u^{*}= 10$. If we can show that $T_0$ is already in the order of $|\Gamma_1|^{-1}$ when $|\Gamma_1|\rightarrow 0$, then one can expect the order for $T^{*}$ should be $|\Gamma_1|^{-1}$, since it is the same as the order of the upper bound. 

In this section, we define $M_1, m_1:[0,T^{*})\rightarrow\m{R}$ by
\[M_1(t)=\max_{x\in\ol{\O}}u^{*}(x,t)\]
and 
\[m_1(t)=\min_{x\in\ol{\O}}u^{*}(x,t).\]
The following are the simulation results by applying the Finite Difference Method. 
\begin{itemize}
\item 2 Dimension: Unit square, space step size h=1/40, time step size 
$k=0.2h^2$, the length $|\Gamma_1|$ from 20/40 to 3/40;
\item 3 Dimension: Unit cubic, space step size h=1/10, time step size 
$k=0.1h^2$, the area $|\Gamma_1|$ from 49/100 to 9/100.
\end{itemize}
Two Dimensional Cases:
\begin{itemize}
\item Let $q=2$ and initial data $u_0(x)\equiv 0.05$. Table \ref{table 1} denotes the first time $T_0$ for 
$M_1$ to reach 10.
\begin{table}[!ht]
	\begin{center}
		\begin{tabular}{|c|c|c|c|c|} \hline
			$|\Gamma_1|$ & 20/40 & 10/40 & 5/40 & 3/40 \\ \hline
			$T_0$ & 35.4 & 72.8 & 149.6 & 253.6 \\ \hline			
			Order & & {1.040} & 
				{1.039} & {1.033}\\ \hline		
			$m_1(T_0)$ & 1.17 & 1.57 & 2.32 & 3.21 \\ \hline
		\end{tabular}	 
	\end{center} 
	\caption{2 Dimension, $q=2$}
	\label{table 1}
\end{table}

\item Let $q=3$ and initial data $u_0(x)\equiv 0.05$. Table \ref{table 2} denotes the first time $T_0$ for 
$M_1$ to reach 10.
\begin{table}[!ht]
	\begin{center}
		\begin{tabular}{|c|c|c|c|c|} \hline
			$|\Gamma_1|$ & 20/40 & 10/40 & 5/40 & 3/40 \\ \hline
			$T_0$ & 394.6 & 791.8 & 1588.5 & 2652.5 \\ \hline		
			Order & & {1.005} & 
				{1.005} & {1.004}\\ \hline		
			$m_1(T_0)$ & 0.81 & 0.95 & 1.16 & 1.37 \\ \hline
		\end{tabular}	 
	\end{center} 
	\caption{2 Dimension, $q=3$}
	\label{table 2}
\end{table}
\end{itemize}
Three Dimensional Cases:
\begin{itemize}
\item Let $q=2$ and initial data $u_0(x)\equiv 0.05$. Table \ref{table 3} denotes the first time $T_0$ for 
$M_1$ to reach 10.
\begin{table}[!ht]
	\begin{center}
		\begin{tabular}{|c|c|c|c|c|} \hline
			$|\Gamma_1|$ & 49/100 & 25/100 & 16/100 & 9/100 \\ \hline
			$T_0$ & 36.3 & 72.6 & 114.7 & 206.9 \\ \hline			
			Order & & {1.028} & 
				{1.024} & {1.026}\\ \hline		
			$m_1(T_0)$ & 1.23 & 1.51 & 1.73 & 2.17 \\ \hline
		\end{tabular}	 
	\end{center} 
	\caption{3 Dimension, $q=2$}
	\label{table 3}
\end{table}

\item Let $q=3$ and initial data $u_0(x)\equiv 0.05$.
Table \ref{table 4} denotes the first time $T_0$ for 
$M_1$ to reach 10.
\begin{table}[!ht]
	\begin{center}
		\begin{tabular}{|c|c|c|c|c|} \hline
			$|\Gamma_1|$ & 49/100 & 25/100 & 16/100 & 9/100 \\ 		\hline
			$T_0$ & 403.0 & 791.6 & 1238.4 & 2205.4 \\ \hline		
			Order & & {1.003} & 
				{1.003} & {1.003}\\ \hline		
			$m_1(T_0)$ & 0.84 & 0.93 & 1.00 & 1.13 \\ \hline
		\end{tabular}	 
	\end{center} 
	\caption{3 Dimension, $q=3$}
	\label{table 4}
\end{table}

\end{itemize}

From any of these tables, the order of the blow-up time is about 1. Thus, we conjecture that the blow-up time $T^{*}$ should be comparable to $|\Gamma_1|^{-1}$ as $|\Gamma_1|\rightarrow 0$.
\bigskip
\bigskip

\begin{appendix}
\section*{Appendices}

\section{Jump Relation}
\label{Sec, Jump Relation}
One of the key tools to show the existence of the solution to the parabolic equations with Neumann boundary conditions is the Jump Relation of the Single-layer Potentials (See \cite{Friedman}, Sec. 2, Chap. 5). For the convenience of the readers, we restate it here.
\begin{theorem}[Jump Relation]\label{Thm, Jump Relation, Original}
Let $g\in C\big(\p\O\times[0,T]\big)$, then for any $(x,t)\in\p\O\times(0,T]$, 
\begin{align}\label{Jump Relation, Original}
&\lim_{h\rightarrow 0^{+}}\int_{0}^{t}\int_{\p\O}(D\Phi)\big(x_h-y,t-\tau\big)\cdot\ora{n}(x)\,g(y,\tau)\,dS_{y}\,d\tau \nonumber \\
=&\int_{0}^{t}\int_{\p\O}(D\Phi)(x-y,t-\tau)\cdot\ora{n}(x)\,g(y,\tau)\,dS_{y}\,d\tau+\frac{1}{2}\,g(x,t),
\end{align} 
where $\Phi$ is the fundamental solution of the heat equation, $\ora{n}(x)$ denotes the exterior unit normal vector at $x$ with respect to $\p\O$ and $x_h\triangleq x-h\ora{n}(x)$.
\end{theorem} 

In this section, we discuss several variants of Theorem \ref{Thm, Jump Relation, Original}. These variants will be mainly applied to show the existence of the solution to (\ref{Linear Prob, Simple Model}) and (\ref{Nonlinear Prob, Simple Model}). We also have used them in some other places, for example, in the proofs of Lemma \ref{Lemma, conv. of approx. sol'n} and Corollary \ref{Cor, rep. formula for bdry.}. 

The first variant is the following Corollary \ref{Cor, Jump Relation for vary normal}, in which the normal direction in the integrand is not fixed to be $\ora{n}(x)$ as in Theorem \ref{Thm, Jump Relation, Original}.

\begin{corollary}\label{Cor, Jump Relation for vary normal}
Let $g\in C\big(\p\O\times[0,T]\big)$, then for any $x\in\p\O$, $t\in(0,T]$,
\begin{align}\label{Jump Relation for vary normal}
&\lim_{h\rightarrow 0^{+}}\int_{0}^{t}\int_{\p\O}(D\Phi)\big(x_h-y,t-\tau\big)\cdot\ora{n}(y)\,g(y,\tau)\,dS_{y}\,d\tau \nonumber \\
=&\int_{0}^{t}\int_{\p\O}(D\Phi)(x-y,t-\tau)\cdot\ora{n}(y)\,g(y,\tau)\,dS_{y}\,d\tau+\frac{1}{2}\,g(x,t).
\end{align} 
\end{corollary} 
\begin{proof}
Based on Theorem \ref{Thm, Jump Relation, Original}, it suffices to show
\begin{align}\label{small difference in jump relation}
&\lim_{h\rightarrow 0^{+}}
\int_{0}^{t}\int_{\p\O}(D\Phi)\big(x_h-y,t-\tau\big)\cdot \big[\ora{n}(y)-\ora{n}(x)\big]\,g(y,\tau)\,dS_{y}\,d\tau 
\nonumber \\
=&\int_{0}^{t}\int_{\p\O}(D\Phi)(x-y,t-\tau)\cdot \big[\ora{n}(y)-\ora{n}(x)\big]\,g(y,\tau)\,dS_{y}\,d\tau.
\end{align}
We denote $N_{x}$ to be the normal line of $\p\O$ at $x$. Since $\p\O\in C^{2}$, then it satisfies the interior ball condition at $x$, which means that there exists an open ball $B\subset\O$ with center on $N_{x}$ and $\bar{B}\cap\p\O=\{x\}$. Hence, if we denote the radius of $B$ to be $R$, then for any $h\in(0,R)$, $x_h\in N_{x}\cap\O$. In addition, $|x_h-x|\leq |x_h-y|$ for any $y\in\p\O$. Thus, \[|x-y|\leq |x-x_h|+|x_h-y|\leq 2\,|x_h-y|.\]
Now combining the fact $\p\O\in C^{2}$ again, we get
\[\big|\ora{n}(y)-\ora{n}(x)\big|\leq C|y-x|\leq C\,|x_h-y|.\] 
Consequently,
\begin{align*}
&\Big|(D\Phi)\big(x_h-y,t-\tau\big)\cdot \big[\ora{n}(y)-\ora{n}(x)\big]\,g(y,\tau)\Big|\\
\leq\; & \frac{C\,|x_h-y|^2}{(t-\tau)^{n/2+1}}\,\exp\bigg({-\frac{|x_h-y|^2}{4(t-\tau)}}\bigg) \\
\leq\; & \frac{C}{(t-\tau)^{n/2}}\,\exp\bigg({-\frac{|x_h-y|^2}{8(t-\tau)}}\bigg) \\
\leq\; & \frac{C}{(t-\tau)^{n/2}}\,\exp\bigg({-\frac{|x-y|^2}{32(t-\tau)}}\bigg)
\end{align*}
This inequality enables us to apply the Lebesgue's dominated convergence theorem to verify (\ref{small difference in jump relation}).

\end{proof}

Theorem \ref{Thm, Jump Relation, Original} and Corollary \ref{Cor, Jump Relation for vary normal} are still not enough for our problems. For example, in order to show the existence of the solution to (\ref{Linear Prob, Simple Model}), the boundary functions $\beta$ and $g$ are only assumed in $\mathcal{B}_{T}$, not in $C\big(\p\O\times[0,T]\big)$. Thus we need to adapt this formula to the functions in $\mathcal{B}_{T}$. The following Theorem \ref{Thm, Jump Relation, Adapted} is our observation, but the essence of the proof is the same as that of Theorem \ref{Thm, Jump Relation, Original}. 

The following are some notations needed in the proof of Theorem \ref{Thm, Jump Relation, Adapted}. We write ${\bf 0}$ and ${\bf \t{0}}$ to be the origins in $\m{R}^{n}$ and $\m{R}^{n-1}$ respectively and ${\bf e_n}$ denotes the point $(0,0,\cdots,0,1)$ in $\m{R}^{n}$. For any point $y=(y_1,y_2,\cdots,y_n)\in\m{R}^{n}$, we write
\[\tilde{y}=(y_1,y_2,\cdots,y_{n-1}).\] 
For any $r>0$, 
\[B_r\triangleq B({\bf 0},r)\]
means the ball in $\m{R}^{n}$ with radius $r$ and 
\[\wt{B}_r\triangleq B({\bf \t{0}}, r)\]
represents the ball in $\m{R}^{n-1}$ with radius $r$. $\Gamma$ is used to denote the Gamma function, i.e. $\Gamma:(0,\infty)\rightarrow\m{R}$ defined by $$\Gamma(a)=\int_{0}^{\infty}t^{a-1}e^{-t}\,dt.$$
We should be careful to distinguish the notation $\Gamma$ from the notations $\Gamma_1$, $\Gamma_2$ and $\wt{\Gamma}$.

\begin{theorem}\label{Thm, Jump Relation, Adapted}
Let $\varphi\in \mathcal{B}_{T}$, $i=1$ or $2$, then for any $x\in \wt{\Gamma}\triangleq \ol{\Gamma}_1\cap\ol{\Gamma}_2$, $t\in (0,T]$, 
\begin{align}\label{Jump Relation, Adapted}
&\lim_{h\rightarrow 0^{+}}\int_{0}^{t}\int_{\Gamma_i}(D\Phi)\big(x_h-y,t-\tau\big)\cdot\ora{n}(x)\,\varphi(y,\tau)\,dS_{y}\,d\tau \nonumber \\
=&\int_{0}^{t}\int_{\Gamma_i}(D\Phi)(x-y,t-\tau)\cdot\ora{n}(x)\,\varphi(y,\tau)\,dS_{y}\,d\tau+\frac{1}{4}\,\varphi_i(x,t),
\end{align}
where $\varphi_i\triangleq\varphi|_{\Gamma_{i}\times(0,T]}$ represents the extension of $\varphi$ on $\ol{\Gamma}_{i}\times[0,T]$.
\end{theorem}

\begin{proof}
We assume $i=1$ (The case $i=2$ is similar). By (\ref{Fund. Soln.}), (\ref{Jump Relation, Adapted}) becomes 
\begin{align}\label{detailed general jump relation}
&\lim_{h\rightarrow 0^{+}}-\int_{0}^{t}\int_{\Gamma_1}\,\frac{(x_h-y)\cdot\ora{n}(x)}{(t-\tau)^{n/2+1}}\,\exp\bigg(-\frac{|x_h-y|^2}{4(t-\tau)}\bigg)\,\varphi_1(y,\tau)\,dS_{y}\,d\tau \nonumber\\
=&-\int_{0}^{t}\int_{\Gamma_1}\,\frac{(x-y)\cdot\ora{n}(x)}{(t-\tau)^{n/2+1}}\,\exp\bigg(-\frac{|x-y|^2}{4(t-\tau)}\bigg)\,\varphi_1(y,\tau)\,dS_{y}\,d\tau+\frac{(4\pi)^{n/2}}{2}\,\varphi_1(x,t).
\end{align}
Without loss of generality, we assume $x={\bf 0}$, otherwise we can do a translation. After this, we further assume $\ora{n}({\bf 0})=-{\bf e_n}$, otherwise we can do a rotation which preserves the dot product and the distance. By these two simplifications, we have $x={\bf 0}$ and $\ora{n}(x)=-{\bf e_n}$, therefore $x_h=h{\bf e_n}$ and (\ref{detailed general jump relation}) is reduced to 
\begin{align*}
&\lim_{h\rightarrow 0^{+}}\int_{0}^{t}\int_{\Gamma_1}\,\frac{h-y_n}{(t-\tau)^{n/2+1}}\,\exp\bigg(-\frac{|h{\bf e_n}-y|^2}{4(t-\tau)}\bigg)\,\varphi_1(y,\tau)\,dS_{y}\,d\tau \\
=& -\int_{0}^{t}\int_{\Gamma_1}\,\frac{y_n}{(t-\tau)^{n/2+1}}\,\exp\bigg(-\frac{|y|^2}{4(t-\tau)}\bigg)\,\varphi_1(y,\tau)\,dS_{y}\,d\tau+\frac{(4\pi)^{n/2}}{2}\,\varphi_1({\bf 0},t).
\end{align*} 
By a change of variable in $\tau$, it is equivalent to
\begin{align}\label{simplified jump relation}
&\lim_{h\rightarrow 0^{+}}\int_{0}^{t}\int_{\Gamma_1}\,\frac{h-y_n}{\tau^{n/2+1}}\,\exp\bigg(-\frac{|h{\bf e_n}-y|^2}{4\tau}\bigg)\,\varphi_1(y,t-\tau)\,dS_{y}\,d\tau \nonumber\\
=& -\int_{0}^{t}\int_{\Gamma_1}\,\frac{y_n}{\tau^{n/2+1}}\,\exp\bigg(-\frac{|y|^2}{4\tau}\bigg)\,\varphi_1(y,t-\tau)\,dS_{y}\,d\tau+\frac{(4\pi)^{n/2}}{2}\,\varphi_1({\bf 0},t).
\end{align} 

Because $\p\O\in C^{2}$ and $\wt{\Gamma}=\p\Gamma_1\in C^{1}$, we can straighten the boundary. More specifically, after relabeling the coordinates, there exist $\phi_1\in C^{2}:\m{R}^{n-1}\rightarrow\m{R}$, $\phi_2\in C^{1}:\m{R}^{n-2}\rightarrow\m{R}$, $\eta_0>0$ and a neighborhood $S_{\eta_0}\subset\p\O$ of ${\bf 0}$ such that $S_{\eta_0}$ can be parametrized as 
\[S_{\eta_0}=\{\big(\t{y},\phi_1(\t{y})\big):\t{y}\in \wt{B}_{\eta_0}\}\]
and for any $y\in \wt{\Gamma}\cap S_{\eta_0}$, we not only have $y_n=\phi_1(\t{y})$, but also $y_{n-1}=\phi_2(y_1,y_2,\cdots,y_{n-2})$. Fixing $\eta_0$ and for any $\eta<\eta_0$, we define 
\[S_{\eta}=\{\big(\t{y},\phi_1(\t{y})\big):\t{y}\in \wt{B}_{\eta}\},\]
which is a subset of $S_{\eta_0}$ and also a small neighborhood of ${\bf 0}$. Then we denote \[S_{\eta,1}=S_{\eta}\cap\Gamma_1,\quad \wt{S}_{\eta}=S_{\eta}\cap\wt{\Gamma},\]
\[\wt{B}_{\eta,1}=\{\t{y}\in\wt{B}_{\eta}:\big(\t{y},\phi_1(\t{y})\big)\in S_{\eta,1}\},\quad P_{\eta}=\{\t{y}\in\wt{B}_{\eta}:\big(\t{y},\phi_1(\t{y})\big)\in \wt{S}_{\eta}\}.\]
After these preparations, we begin the technical proof. Given any $\v>0$, we want to find $\delta=\delta(\v)>0$ such that for any $0<h<\delta$, the difference between the two sides of (\ref{simplified jump relation}) is within $C\v$ for some constant $C$.

For any $\eta\in(0,\eta_0)$ which will be determined later, we split the integral over $\Gamma_1$ in (\ref{simplified jump relation}) into two parts: $\int_{\Gamma_1}=\int_{S_{\eta,1}}+\int_{\Gamma_1\setminus S_{\eta,1}}$. Since $\Gamma_1\setminus S_{\eta,1}$ is away from ${\bf 0}$, it is easy to see there exists $\delta_1=\delta_1(\eta,\v)$ such that when $0<h<\delta_1$, then 
\be\label{outside part estimate} \begin{split}
&\bigg|\int_{0}^{t}\int_{\Gamma_1\setminus S_{\eta,1}}\,\frac{h-y_n}{\tau^{n/2+1}}\,\exp\bigg(-\frac{|h{\bf e_n}-y|^2}{4\tau}\bigg)\,\varphi_1(y,t-\tau)\,dS_{y}\,d\tau\\
&+\int_{0}^{t}\int_{\Gamma_1\setminus S_{\eta,1}}\,\frac{y_n}{\tau^{n/2+1}}\,\exp\bigg(-\frac{|y|^2}{4\tau}\bigg)\,\varphi_1(y,t-\tau)\,dS_{y}\,d\tau\bigg|<\v.
\end{split}\ee

Next since $\ora{n}({\bf 0})=-{\bf e_n}$, then $D\phi_1(\t{{\bf 0}})=\t{{\bf 0}}$. As a result, for any $y\in S_{\eta,1}$, 
\be\label{straighten estimate}\begin{split}
&|y_n|=|\phi_1(\t{y})|=|\phi_1(\t{y})-\phi_1(\t{{\bf 0}})|=|D\phi_1(\theta\t{y})\cdot\t{y}|\\
\leq & |D\phi_1(\theta\t{y})-D\phi_1(\t{{\bf 0}})|\,|\t{y}|\leq C\,|\t{y}|^{2},
\end{split}\ee
where we used the mean value theorem twice. By (\ref{straighten estimate}), together with the fact $|h{\bf e_n}-y|\geq |\t{y}|$, we attain
\[|y_n|\,\exp\bigg(-\frac{|h{\bf e_n}-y|^2}{4\tau}\bigg)\leq C\,|\t{y}|^2\,\exp\bigg(-\frac{|\t{y}|^2}{4\tau}\bigg).\] 
Noticing \[\int_{0}^{t}\int_{S_{\eta,1}}\frac{|\t{y}|^{2}}{\tau^{n/2+1}}\exp\bigg(-\frac{|\t{y}|^2}{4\tau}\bigg)\,dS_{y}\,d\tau<\infty,\]
then it follows from Lebesgue's dominated convergence theorem that 
\begin{align*} 
& \lim_{h\rightarrow 0^{+}}\int_{0}^{t}\int_{S_{\eta,1}}\,\frac{y_n}{\tau^{n/2+1}}\,\exp\bigg(-\frac{|h{\bf e_n}-y|^2}{4\tau}\bigg)\,\varphi_1(y,t-\tau)\,dS_{y}\,d\tau \\
=&\int_{0}^{t}\int_{S_{\eta,1}}\,\frac{y_n}{\tau^{n/2+1}}\,\exp\bigg(-\frac{|y|^2}{4\tau}\bigg)\,\varphi_1(y,t-\tau)\,dS_{y}\,d\tau.
\end{align*}
As a result, there exists $\delta_2=\delta_2(\eta,\v)$ such that when $0<h<\delta_2$, then 
\be\label{inner convergence part estimate}\begin{split}
&\bigg|\int_{0}^{t}\int_{S_{\eta,1}}\,\frac{y_n}{\tau^{n/2+1}}\,\exp\bigg(-\frac{|h{\bf e_n}-y|^2}{4\tau}\bigg)\,\varphi_1(y,t-\tau)\,dS_{y}\,d\tau \\
&-\int_{0}^{t}\int_{S_{\eta,1}}\,\frac{y_n}{\tau^{n/2+1}}\,\exp\bigg(-\frac{|y|^2}{4\tau}\bigg)\,\varphi_1(y,t-\tau)\,dS_{y}\,d\tau\bigg|<\v.
\end{split}\ee

Now it suffices to verify that $|I_{\eta}(h,t)-\frac{1}{2}\,(4\pi)^{n/2}\,\varphi_{1}({\bf 0},t)|<C\v$, where
\be\label{inner jump part, surface integral}
I_{\eta}(h,t)\triangleq\int_{0}^{t}\int_{S_{\eta,1}}\,\frac{h}{\tau^{n/2+1}}\,\exp\bigg(-\frac{|h{\bf e_n}-y|^2}{4\tau}\bigg)\,\varphi_1(y,t-\tau)\,dS_{y}\,d\tau.
\ee
Recalling that $y_n=\phi(\t{y})$, (\ref{inner jump part, surface integral}) can be rewritten as 
\be\label{inner jump part, flat integral}\begin{split}
I_{\eta}(h,t)& \triangleq \int_{0}^{t}\int_{\wt{B}_{\eta,1}}\,\frac{h}{\tau^{n/2+1}}\,\exp\bigg(-\frac{|h{\bf e_n}-y|^2}{4\tau}\bigg)\,\varphi_1(y,t-\tau)\sqrt{1+|D\phi_1(\t{y})|^2}\,d\t{y}\,d\tau\\
&=\int_{0}^{t}\int_{\wt{B}_{\eta,1}}\,\frac{h}{\tau^{n/2+1}}\,\exp\bigg(-\frac{|\t{y}|^2+|h-y_n|^2}{4\tau}\bigg)\,\varphi_1(y,t-\tau)\sqrt{1+|D\phi_1(\t{y})|^2}\,d\t{y}\,d\tau,
\end{split}\ee
where $y=\big(\t{y},\phi_1(\t{y})\big)$. $I_{\eta}$ is hard to compute, so we approximate it by a simpler function. We define $\t{I}_{\eta}(h,t)$ as following
\be\label{inner jump part, approx.}\begin{split}
\t{I}_{\eta}(h,t)& \triangleq \int_{0}^{t}\int_{\wt{B}_{\eta,1}}\,\frac{h}{\tau^{n/2+1}}\,\exp\bigg(-\frac{|h{\bf e_n}-(\t{y},0)|^2}{4\tau}\bigg)\,\varphi_1({\bf 0},t-\tau)\,d\t{y}\,d\tau\\
&=\int_{0}^{t}\int_{\wt{B}_{\eta,1}}\,\frac{h}{\tau^{n/2+1}}\,\exp\bigg(-\frac{|\t{y}|^2+h^2}{4\tau}\bigg)\,\varphi_1({\bf 0},t-\tau)\,d\t{y}\,d\tau.
\end{split}\ee
Our strategy is to show that $\t{I}_{\eta}(h,t)$ is close to both $\frac{1}{2}\,(4\pi)^{n/2}\,\varphi_{1}({\bf 0},t)$ and $I_{\eta}(h,t)$.

Based on (\ref{inner jump part, approx.}), we first reverse the order of integration and then make the change of variable $\tau\rightarrow\sigma=\big(|\t{y}|^2+h^2\big)/(4\tau)$ to obtain
\begin{align}\label{inner jump part, approx, compact form}
\t{I}_{\eta}(h,t) &\triangleq\int_{\wt{B}_{\eta,1}}\,\int_{0}^{t}\frac{h}{\tau^{n/2+1}}\,\exp\bigg(-\frac{|\t{y}|^2+h^2}{4\tau}\bigg)\,\varphi_1({\bf 0},t-\tau)\,d\tau\,d\t{y} \nonumber\\
&= \int\limits_{\wt{B}_{\eta,1}}\,\int\limits_{\frac{|\t{y}|^2+h^2}{4t}}^{\infty}\frac{4^{n/2}\,\sigma^{n/2-1}\,h}{(|\t{y}|^2+h^2)^{n/2}}\,e^{-\sigma}\,\varphi_1\Big({\bf 0},t-\frac{|\t{y}|^2+h^2}{4\sigma}\Big)\,d\sigma\,d\t{y} 
\nonumber\\
&= 4^{n/2}\int\limits_{\wt{B}_{\eta,1}}\frac{h}{(|\t{y}|^2+h^2)^{n/2}}\int\limits_{\frac{|\t{y}|^2+h^2}{4t}}^{\infty}\sigma^{n/2-1}e^{-\sigma}\,\varphi_1\Big({\bf 0},t-\frac{|\t{y}|^2+h^2}{4\sigma}\Big)\,d\sigma\,d\t{y} \nonumber\\
&= 4^{n/2}\int_{\wt{B}_{\eta,1}}\frac{h}{(|\t{y}|^2+h^2)^{n/2}}\,H\big(|\t{y}|^2+h^2,t\big)\,d\t{y},
\end{align}
where 
\[H(\lambda,t)\triangleq\int_{\frac{\lambda}{4t}}^{\infty}\sigma^{n/2-1}e^{-\sigma}\,\varphi_1\Big({\bf 0},t-\frac{\lambda}{4\sigma}\Big)\,d\sigma.\]
It is readily to see that
\be\label{limit of H function}
\lim_{\lambda\rightarrow 0}H(\lambda,t)=\Gamma\Big(\frac{n}{2}\Big)\,\varphi_1({\bf 0},t).
\ee
Consequently, there exists $\delta_3=\delta_3(\v)$ such that when $\eta<\delta_3$ and $0<h<\delta_3$, then 
\be\label{estimate for H function}
\Big|H\big(|\t{y}|^2+h^2,t\big)-\Gamma\Big(\frac{n}{2}\Big)\varphi_1({\bf 0},t)\Big|<\v, \quad\forall\, \t{y}\in \wt{B}_{\eta,1}.
\ee

After having taken care of the $H$ term in (\ref{inner jump part, approx, compact form}), let's consider the following integration
\[\int\limits_{\wt{B}_{\eta,1}}\frac{h}{(|\t{y}|^2+h^2)^{n/2}}\,d\t{y},\]
where the integrand $h\,\big(|\t{y}|^2+h^2\big)^{-n/2}$ is radian in $\t{y}\in\m{R}^{n-1}$ and positive when $h>0$. Since $\wt{\Gamma}=\p\Gamma_1\in C^{1}$, it ensures that $P_{\eta}$ almost bisects $\wt{B}_{\eta}$ when $\eta$ is  small, which means $\wt{B}_{\eta,1}$ is close to a hemisphere and
\[\lim_{\eta\rightarrow 0}\frac{|\wt{B}_{\eta,1}|}{|\wt{B}_{\eta}|}=\frac{1}{2}.\] As a result, we can find $\delta_4=\delta_4(\v)$ such that for any $\eta<\delta_4$, 
\[1-\v<\frac{\int\limits_{\wt{B}_{\eta,1}}\frac{h}{(|\t{y}|^2+h^2)^{n/2}}\,d\t{y}}{\frac{1}{2}\int\limits_{\wt{B}_{\eta}}\frac{h}{(|\t{y}|^2+h^2)^{n/2}}\,d\t{y}}<1+\v,\]
i.e. 
\be\label{estimate the integral by half of the integral over a whole ball}
\bigg|\int\limits_{\wt{B}_{\eta,1}}\frac{h}{(|\t{y}|^2+h^2)^{n/2}}\,d\t{y}-\frac{1}{2}\int\limits_{\wt{B}_{\eta}}\frac{h}{(|\t{y}|^2+h^2)^{n/2}}\,d\t{y}\bigg|<\frac{\v}{2}\int\limits_{\wt{B}_{\eta}}\frac{h}{(|\t{y}|^2+h^2)^{n/2}}\,d\t{y}. \ee
Next, we will estimate $\int\limits_{\wt{B}_{\eta}}\frac{h}{(|\t{y}|^2+h^2)^{n/2}}\,d\t{y}$.
Making the change of variable $\t{y}\rightarrow \t{z}\triangleq\t{y}/h$,  
\bes \int\limits_{\wt{B}_{\eta}}\frac{h}{(|\t{y}|^2+h^2)^{n/2}}\,d\t{y}=\int\limits_{\wt{B}_{\eta/h}}\frac{1}{(|\t{z}|^2+1)^{n/2}}\,d\t{z}. \ees
On one hand, 
\[\int\limits_{\wt{B}_{\eta/h}}\frac{1}{(|\t{z}|^2+1)^{n/2}}\,d\t{z}\leq \int\limits_{\m{R}^{n-1}}\frac{1}{(|\t{z}|^2+1)^{n/2}}\,d\t{z}=\frac{\pi^{n/2}}{\Gamma\big(\frac{n}{2}\big)},\]
while on the other hand, 
\[\lim_{h\rightarrow 0}\int\limits_{\wt{B}_{\eta/h}}\frac{1}{(|\t{z}|^2+1)^{n/2}}\,d\t{z}=\int\limits_{\m{R}^{n-1}}\frac{1}{(|\t{z}|^2+1)^{n/2}}\,d\t{z}=\frac{\pi^{n/2}}{\Gamma\big(\frac{n}{2}\big)}.\]
Thus, there exists $\delta_5=\delta_5(\eta,\v)$ such that for any $0<h<\delta_5$,  
\be\label{the value of the integral over the whole ball}
\bigg|\int\limits_{\wt{B}_{\eta}}\frac{h}{(|\t{y}|^2+h^2)^{n/2}}\,d\t{y}-\frac{\pi^{n/2}}{\Gamma\big(\frac{n}{2}\big)}\bigg|<\v \ee
and therefore by (\ref{estimate the integral by half of the integral over a whole ball}),
\be\label{estimate the targeted integral}
\bigg|\int\limits_{\wt{B}_{\eta,1}}\frac{h}{(|\t{y}|^2+h^2)^{n/2}}\,d\t{y}-\frac{\pi^{n/2}}{2\Gamma\big(\frac{n}{2}\big)}\bigg|<C\v. \ee
It then follows from (\ref{inner jump part, approx, compact form}), (\ref{estimate for H function}) and (\ref{estimate the targeted integral}) that 
\be\label{estimate for the approximated integral}
\Big|\t{I}_{\eta}(h,t)-\frac{(4\pi)^{n/2}}{2}\varphi_1({\bf 0},t)\Big|<C\v. \ee

Now it left to show that $\t{I}_{\eta}(h,t)$ is close to $I_{\eta}(h,t)$. Firstly, because of (\ref{straighten estimate}), $|\t{y}|^2+|h-y_n|^2$ is comparable to $|\t{y}|^2+h^2$. More precisely, there exist positive constants $m_1<1$ and $M_{1}>1$ such that 
\be\label{distance comparable}
m_1\,\big(|\t{y}|^2+h^2\big)\leq |\t{y}|^2+|h-y_n|^2 \leq M_{1}\,\big(|\t{y}|^2+h^2\big). \ee
We can equivalently write it to be
\be\label{distance comparable, compact form}
m_1\,|h{\bf e_n}-(\t{y},0)|^2\leq |h{\bf e_n}-y|^2\leq M_{1}\,|h{\bf e_n}-(\t{y},0)|^2. \ee
Next, it follows from (\ref{inner jump part, flat integral}) and (\ref{inner jump part, approx.}) that
\begin{align}
&|I_{\eta}(h,t)-\t{I}_{\eta}(h,t)| \nonumber\\
\leq & \int_{0}^{t}\frac{h}{\tau^{n/2+1}}\int_{\wt{B}_{\eta,1}}\Big|e^{-\frac{|h{\bf e_n}-y|^2}{4\tau}}-e^{-\frac{|h{\bf e_n}-(\t{y},0)|^2}{4\tau}}\Big|\,\Big|\varphi_1(y,t-\tau)\Big|\sqrt{1+|D\phi_1(\t{y})|^2}\,d\t{y}\,d\tau \nonumber\\
&+\int_{0}^{t}\frac{h}{\tau^{n/2+1}}\int_{\wt{B}_{\eta,1}}e^{-\frac{|h{\bf e_n}-(\t{y},0)|^2}{4\tau}}\Big|\varphi_1(y,t-\tau)\sqrt{1+|D\phi_1(\t{y})|^2}-\varphi_1({\bf 0}, t-\tau)\Big|\,d\t{y}\,d\tau \nonumber\\
\triangleq & I+II,
\end{align}
where $y=\big(\t{y},\phi_1(\t{y})\big)$. For $II$, since $\varphi_1\in C\big(\ol{\Gamma}_1\times[0,T]\big)$, then there exists $\delta_6=\delta_6(\v)$ such that when $\eta<\delta_6$,
\[\Big|\varphi_1(y,t-\tau)\sqrt{1+|D\phi_1(\t{y})|^2}-\varphi_1({\bf 0}, t-\tau)\Big|<\v, \quad\forall\,\t{y}\in \wt{B}_{\eta,1},\,\tau\in[0,t].\]
As a result, 
\begin{align*}
II &\leq \v\int_{0}^{t}\frac{h}{\tau^{n/2+1}}\int_{\wt{B}_{\eta,1}}e^{-\frac{|h{\bf e_n}-(\t{y},0)|^2}{4\tau}}\,d\t{y}\,d\tau\\
&=\v\int_{0}^{t}\frac{h}{\tau^{n/2+1}}\int_{\wt{B}_{\eta,1}}e^{-\frac{|\t{y}|^2+h^2}{4\tau}}\,d\t{y}\,d\tau\\
&=\v\int_{\wt{B}_{\eta,1}}h\int_{0}^{t}\frac{1}{\tau^{n/2+1}}\,e^{-\frac{|\t{y}|^2+h^2}{4\tau}}\,d\tau\,d\t{y}\\
&\leq \v\int_{\wt{B}_{\eta,1}}h\int_{0}^{\infty}\frac{4^{n/2}}{\big(|\t{y}|^2+h^2\big)^{n/2}}\,\sigma^{n/2-1}\,e^{-\sigma}\,d\sigma\,d\t{y}\\
&=C\,\v\int_{\wt{B}_{\eta,1}}\frac{h}{\big(|\t{y}|^2+h^2\big)^{n/2}}\,d\t{y},
\end{align*}
where the second inequality is due to the change of variable 
$\tau\rightarrow\sigma\triangleq \frac{|\t{y}|^2+h^2}{4\tau}$. Now by another change of variabel $\t{y}\rightarrow\t{z}\triangleq\t{y}/h$, we get
\begin{align}\label{estimate for the second part of the diff.}
II &\leq C\,\v\int_{\m{R}^{n-1}}\frac{1}{(|\t{z}|^2+1)^{n/2}}\,d\t{z} =C\v.
\end{align}
To estimate $I$, firstly it is easy to see that for any $h>0$ and $y\in\wt{B}_{\eta,1}$,
\begin{align}\label{estimate for h}
h\leq |h{\bf e_n}-(\t{y},0)|.
\end{align}
Then by (\ref{straighten estimate}),
\begin{align*}
&\Big||h{\bf e_n}-y|^2-|h{\bf e_n}-(\t{y},0)|^2\Big|\\
=\, & \big|(h-y_n)^2-h^2\big|\\
=\, & |y_n|\,|2h-y_n|\\
\leq\, & C\,|\t{y}|^2\big(2h+|\t{y}|^2\big)\\
\leq\, & C\,|h{\bf e_n}-(\t{y},0)|^3.
\end{align*}
Now it follows from the mean value theorem and (\ref{distance comparable, compact form}) that
\begin{align}\label{difference in exp. func.}
\Big|e^{-\frac{|h{\bf e_n}-y|^2}{4\tau}}-e^{-\frac{|h{\bf e_n}-(\t{y},0)|^2}{4\tau}}\Big|&\leq \frac{1}{4\tau}\,e^{-\frac{m_1\,|h{\bf e_n}-(\t{y},0)|^2}{4\tau}}\,\Big||h{\bf e_n}-y|^2-|h{\bf e_n}-(\t{y},0)|^2\Big| \nonumber\\
&\leq C\,\frac{|h{\bf e_n}-(\t{y},0)|^3}{\tau}\,e^{-\frac{m_1\,|h{\bf e_n}-(\t{y},0)|^2}{4\tau}}
\end{align}
Thus, based on (\ref{estimate for h}) and (\ref{difference in exp. func.}), we attain
\[I\leq \int_{0}^{t}\int_{\wt{B}_{\eta,1}}\frac{1}{\tau^{n/2+2}}\,e^{-\frac{m_1\,|h{\bf e_n}-(\t{y},0)|^2}{4\tau}}\,|h{\bf e_n}-(\t{y},0)|^4\,d\t{y}\,d\tau.\]
Again, by reversing the order of integration and the change of variable $\tau\rightarrow\sigma\triangleq \frac{|h{\bf e_n}-(\t{y},0)|^2}{4\tau}$, we get
\begin{align*}
I&\leq \int_{\wt{B}_{\eta,1}}|h{\bf e_n}-(\t{y},0)|^4\int_{0}^{t}\frac{1}{\tau^{n/2+2}}\,e^{-\frac{m_1\,|h{\bf e_n}-(\t{y},0)|^2}{4\tau}}\,d\tau\,d\t{y} \nonumber\\
& \leq C\int_{\wt{B}_{\eta,1}}\frac{1}{|h{\bf e_n}-(\t{y},0)|^{n-2}}\int_{0}^{\infty}\sigma^{n/2}\,e^{-m_1\sigma}\,d\sigma\,d\t{y}\nonumber\\
&\leq C\int_{\wt{B}_{\eta,1}}\frac{1}{|\t{y}|^{n-2}}\,d\t{y}.
\end{align*}
Hence, there exists $\delta_7=\delta_7(\v)$ such that when $\eta<\delta_7$, then 
\be\label{estimate for the first part of the diff.}
I<\v. \ee
Combining (\ref{estimate for the first part of the diff.}) and (\ref{estimate for the second part of the diff.}), we get \[|I_{\eta}(h,t)-\t{I}_{\eta}(h,t)|<C\v.\] Therefore, we finish the proof.

In summary, for any $\v>0$, we firstly determine $\delta_3(\v),\delta_4(\v),\delta_6(\v),\delta_7(\v)$ and  choose\\
$\eta<\min\{\eta_0,\delta_3,\delta_4,\delta_6,\delta_7\}$. Then we determine $\delta_1(\eta,\v),\delta_2(\eta,\v),\delta_5(\eta,\v)$ and choose $\delta<\min\{\delta_1,\delta_2,\delta_3,\delta_5\}$. Such $\delta$ is what we desire, because we can see from the above proof that for any $0<h<\delta$, the difference between the two sides of (\ref{Jump Relation, Adapted}) is less than $C\v$ for some constant $C$. 
\end{proof}

\begin{corollary}\label{Cor, Jump Relation, Adapted trivial}
Let $\varphi\in \mathcal{B}_{T}$, $i=1$ or $2$, then for any $x\in \Gamma_i$, $t\in (0,T]$, 
\begin{align}\label{Jump Relation, Adapted trivial}
&\lim_{h\rightarrow 0^{+}}\int_{0}^{t}\int_{\Gamma_i}(D\Phi)\big(x_h-y,t-\tau\big)\cdot\ora{n}(x)\,\varphi(y,\tau)\,dS_{y}\,d\tau \nonumber \\
=&\int_{0}^{t}\int_{\Gamma_i}(D\Phi)(x-y,t-\tau)\cdot\ora{n}(x)\,\varphi(y,\tau)\,dS_{y}\,d\tau+\frac{1}{2}\,\varphi(x,t).
\end{align}
\end{corollary}
\begin{proof}
The proof is almost the same as that of Theorem \ref{Thm, Jump Relation, Adapted}. The only difference is this time $x$ is an inner point of $\Gamma_{i}$, as a result, the jump term becomes $\frac{1}{2}\,\varphi(x,t)$ instead of $\frac{1}{4}\,\varphi(x,t)$.
\end{proof}

\begin{corollary}\label{Cor, Jump Relation, Adapted to whole bdry}
Let $\varphi\in \mathcal{B}_{T}$, then for any $x\in \Gamma_1\cup\Gamma_2$, $t\in (0,T]$, 
\begin{align*}
&\lim_{h\rightarrow 0^{+}}\int_{0}^{t}\int_{\p\O}(D\Phi)\big(x_h-y,t-\tau\big)\cdot\ora{n}(x)\,\varphi(y,\tau)\,dS_{y}\,d\tau \nonumber \\
=&\int_{0}^{t}\int_{\p\O}(D\Phi)(x-y,t-\tau)\cdot\ora{n}(x)\,\varphi(y,\tau)\,dS_{y}\,d\tau+\frac{1}{2}\,\varphi(x,t).
\end{align*}
\end{corollary}
\begin{proof}
Without loss of generality, we suppose $x\in\Gamma_1$, then by Corollary \ref{Cor, Jump Relation, Adapted trivial}, 
\begin{align*}
&\lim_{h\rightarrow 0^{+}}\int_{0}^{t}\int_{\Gamma_1}(D\Phi)\big(x_h-y,t-\tau\big)\cdot\ora{n}(x)\,\varphi(y,\tau)\,dS_{y}\,d\tau \nonumber \\
=&\int_{0}^{t}\int_{\Gamma_1}(D\Phi)(x-y,t-\tau)\cdot\ora{n}(x)\,\varphi(y,\tau)\,dS_{y}\,d\tau+\frac{1}{2}\,\varphi(x,t).
\end{align*}
In addition, since the distance between $x$ and $\Gamma_2$ is positive, then it is easy to see that 
\begin{align*}
&\lim_{h\rightarrow 0^{+}}\int_{0}^{t}\int_{\Gamma_2}(D\Phi)\big(x_h-y,t-\tau\big)\cdot\ora{n}(x)\,\varphi(y,\tau)\,dS_{y}\,d\tau \nonumber \\
=&\int_{0}^{t}\int_{\Gamma_2}(D\Phi)(x-y,t-\tau)\cdot\ora{n}(x)\,\varphi(y,\tau)\,dS_{y}\,d\tau.
\end{align*}
Adding these two equations together, the Corollary follows.
\end{proof}

\section{Existence and Uniqueness}
\label{Sec, exist. and uniq.}

\subsection{Linear Case}
\label{Subsec, Linear Case}

In this subsection, we will show the existence and uniqueness of the solution to the following linear initial-boundary value problem:
\be\label{Linear Prob, Simple Model}
\left\{\begin{array}{lll}
u_t(x,t)-\Delta u(x,t)=f(x,t) &\text{in}& \O\times(0,T], \\
\frac{\p u}{\p n}(x,t)+\beta(x,t) u(x,t)=g(x,t) &\text{on}& (\Gamma_1\cup\Gamma_2)\times(0,T],\\
u(x,0)=\psi(x) &\text{in}& \O,
\end{array}\right. \ee
where $f\in C^{\alpha,\alpha/2}\big(\ol{\O}\times[0,T]\big)$, $\beta,g\in\mathcal{B}_{T}$, $\psi\in C^{1}(\overline{\O})$.
We will first show the existence and then use the existence to verify the uniqueness. In the following, $\Gamma$ is also used to denote the Gamma function, i.e. $\Gamma:(0,\infty)\rightarrow\m{R}$ defined by $$\Gamma(a)=\int_{0}^{\infty}t^{a-1}e^{-t}\,dt.$$
For any $T>0$, we write
\be\label{domain for the kernel}
D_{\O,T}=\{(x,t;y,\tau)\,\big|\, x,y\in\O, x\neq y, 0\leq\tau<t\leq T\}\ee
to denote the domain of $K_{j}(j\geq 0)$ which will be constructed in the proof of Theorem \ref{Thm, exist. for linear simple model}. The solution to (\ref{Linear Prob, Simple Model}) is understood in the following way.

\begin{definition}\label{Def, soln to Linear Prob.}
For any $T>0$, a solution to (\ref{Linear Prob, Simple Model}) on $\ol{\O}\times[0,T]$ means a function $u$ in $\mathcal{A}_{T}$ that satisfies (\ref{Linear Prob, Simple Model}) pointwise and moreover, for any $(x,t)\in\wt{\Gamma}\times(0,T]$, $\frac{\p u}{\p n}(x,t)$ exists and
\be\label{interface bdry deri. for linear model}
\frac{\p u}{\p n}(x,t)+ \frac{1}{2}\,\big[\beta_1(x,t)+\beta_2(x,t)\big]u(x,t)=\frac{1}{2}\,\big[g_1(x,t)+g_2(x,t)\big], \ee
where $\beta_i$ and $g_i$ denote the extensions of $\beta$ and $g$ on $\ol{\Gamma}_{i}\times[0,T]$ for $i=1$ or $2$. 
\end{definition}

Before showing the existence of the solution, we state some basic properties. 

\begin{lemma}\label{Lemma, geom. prop. of C^2 bdry}
Suppose $\O$ is an open bounded set in $\m{R}^{n}$ with $\p\O\in C^2$, then there exists a constant $C>0$ such that for any $x,y\in\p\O$,
\[|(x-y)\cdot\ora{n}(x)|\leq C\,|x-y|^2.\]
\end{lemma}
\begin{proof}
It is easy to show this conclusion by taking advantage of the definition of a $C^{2}$ boundary.

\end{proof}

\begin{lemma}\label{Lemma, singular bdry integral}
Suppose $\O$ is an open, bounded set in $\m{R}^{n}$ with $\p\O\in C^2$, $0\leq a<n-1$, $0\leq b<n-1$, then there exists a constant $C$, only depending on $a, b, n, \O$, such that for any $x,z\in\p\O$,
\bes \int_{\p\O}\frac{dS_{y}}{|x-y|^{a}|y-z|^{b}}\leq \left\{\begin{array}{ll}
C\,|x-z|^{n-1-a-b} & \text{if}\quad a+b>n-1,\\
C & \text{if}\quad a+b<n-1.
\end{array}\right. \ees
\end{lemma}

\begin{proof}
See (\cite{Friedman}, Lemma 1, Sec. 2, Chap. 5).
\end{proof}

The following Lemma is mentioned in \cite{Friedman} and it is an important technique used in Theorem \ref{Thm, exist. for linear simple model}, Theorem \ref{Thm, exist. for nonlin. simple model} and Lemma \ref{Lemma, conv. of approx. sol'n}.

\begin{lemma}\label{Lemma, property of iterative kernels}
Let $K_0:D_{T,\O}\rightarrow\m{R}$ and suppose there is a constant $C$ such that for any $(x,t;y,\tau)\in D_{T,\O}$, 
\begin{align}\label{est. for the kernel of the induction def for bdry fcn}
|K_0(x,t;y,\tau)| \leq C\,(t-\tau)^{-3/4}\,|x-y|^{-(n-3/2)}.
\end{align}
For any $j\geq 1$, we define $K_j:D_{T,\O}\rightarrow\m{R}$ by
\be\label{induction def for kernel of fix pt.}
K_{j}(x,t;y,\tau)\triangleq \int_{\tau}^{t}\int_{\p\O}K_0(x,t;z,\sigma)\,K_{j-1}(z,\sigma;y,\tau)\,dS_{z}\,d\sigma. \ee
Then all the $K_j(j\geq 1)$ are well-defined and the series $\sum_{j=0}^{\infty}|K_j|$ converges uniformly to some function $\wt{K}$ on $D_{T,\O}$. Moreover, there exists some constant $C^{*}$, only depending on $n$, $\O$ and $T$, such that for any $(x,t;y,\tau)\in D_{T,\O}$,
\be\label{est. for accumulated kernel}
\wt{K}(x,t;y,\tau)\leq C^{*}(t-\tau)^{-3/4}\,|x-y|^{-(n-3/2)}.\ee
\end{lemma}

\begin{proof}
We can mimic the arguments from Page 14 to Page 15 in \cite{Friedman} to prove this Lemma, provided we take advantage of (\ref{est. for the kernel of the induction def for bdry fcn}) and Lemma \ref{Lemma, singular bdry integral}. Also see the proof of Theorem 2 in (\cite{Friedman}, Sec. 3, Chap. 5).
\end{proof}

\begin{lemma}\label{soln to inhomo. eq.}
If $f\in C^{\alpha,\alpha/2}\big(\ol{\O}\times[0,T]\big)$ and 
\[W(x,t)\triangleq \int_{0}^{t}\int_{\O}\Phi(x-y,t-\tau)\,f(y,\tau)\,dS_{y}\,d\tau, \quad\forall\, (x,t)\in\ol{\O}\times[0,T],\]
then $W\in C^{2,1}\big(\O\times(0,T]\big)$ and 
\[(W_t-\Delta W)(x,t)=f(x,t), \quad\forall\, (x,t)\in\O\times(0,T].\]
\end{lemma}
\begin{proof}
See (\cite{Friedman}, Theorem 9, Sec. 5, Chap. 1). 
\end{proof}

Now based on the arguments in (\cite{Friedman}, Theorem 2, Sec. 3, Chap. 5), we can prove the following existence theorem.

\begin{theorem}\label{Thm, exist. for linear simple model}
For any $T>0$, there exists a solution $u\in\mathcal{A}_{T}$ to (\ref{Linear Prob, Simple Model}) on $\ol{\O}\times[0,T]$. 
\end{theorem}
\begin{proof}
We will construct a solution $u$ to (\ref{Linear Prob, Simple Model}). Firstly, since $\psi\in C^{1}(\overline{\O})$ and $\p\O\in C^{2}$, one can extends $\psi$ to a larger domain. More precisely, there exists an open set $\O_1\supset\ol{\O}$ and $\psi_1\in C^{1}(\ol{\O}_1)$ such that $\psi_1$ agrees with $\psi$ on $\ol{\O}$. In the rest of the proof, for convenience, we just write $\psi_1$ to be $\psi$ and therefore $\psi\in C^{1}(\ol{\O}_1)$. We are looking for a solution $u$ in the following form: for any $(x,t)\in \ol{\O}\times[0,T]$,
\be\label{sol'n for simple model, conjectured form}\begin{split}
u(x,t)\triangleq & \int_{\O_1}\Phi(x-y,t)\,\psi(y)\,dy+\int_{0}^{t}\int_{\O}\Phi(x-y,t-\tau)\,f(y,\tau)\,dy\,d\tau\\
&+\int_{0}^{t}\int_{\p\O}\Phi(x-y,t-\tau)\,\varphi(y,\tau)\,dS_{y}\,d\tau, 
\end{split}\ee
where $\varphi\in \mathcal{B}_{T}$ will be determined later. 

Due to Lemma \ref{soln to inhomo. eq.}, it is readily to see that the function $u$ defined in (\ref{sol'n for simple model, conjectured form}) belongs to $\mathcal{A}_{T}$ and satisfies the first and the third equations in (\ref{Linear Prob, Simple Model}), so in order to verify $u$ to be the solution, the only things left to check are
\be\label{bdry cond., linear prob}
\frac{\p u}{\p n}(x,t)+\beta(x,t) u(x,t)=g(x,t), \quad\forall (x,t)\in(\Gamma_1\cup\Gamma_2)\times(0,T] \ee 
and 
\be\label{bdry cond., interface, linear prob}
\frac{\p u}{\p n}(x,t)+ \frac{1}{2}\,\big[\beta_1(x,t)+\beta_2(x,t)\big]u(x,t)=\frac{1}{2}\,\big[g_1(x,t)+g_2(x,t)\big], \quad\forall\, (x,t)\in\wt{\Gamma}\times(0,T].\ee 
The plan is to firstly find a function $\varphi\in\mathcal{B}_{T}$ such that $u$ defined in (\ref{sol'n for simple model, conjectured form}) satisfies (\ref{bdry cond., linear prob}), then we will prove this $u$ satisfies (\ref{bdry cond., interface, linear prob}) as well.

By (\ref{Normal Deri. Def.}) and (\ref{sol'n for simple model, conjectured form}), for any $(x,t)\in (\Gamma_1\cup\Gamma_2)\times(0,T]$,
\begin{align*}
\frac{\p u}{\p n}(x,t) = &\lim_{h\rightarrow 0^{+}} Du\big(x_h,t\big)\cdot\ora{n}(x) \\
=&\lim_{h\rightarrow 0^{+}}\bigg[\int_{\O_1}(D\Phi)(x_h-y,t)\cdot\ora{n}(x)\,\psi(y)\,dy\\
&+\int_{0}^{t}\int_{\O}(D\Phi)(x_h-y,t-\tau)\cdot\ora{n}(x)\,f(y,\tau)\,dy\,d\tau\\
&+\int_{0}^{t}\int_{\p\O}(D\Phi)(x_h-y,t-\tau)\cdot\ora{n}(x)\,\varphi(y,\tau)\,dS_{y}\,d\tau\bigg].
\end{align*}
Applying the Lebesgue's dominated convergence theorem and Corollary \ref{Cor, Jump Relation, Adapted to whole bdry}, we get 
\begin{align}\label{normal deri. for regular bdry pts}
\frac{\p u}{\p n}(x,t)=& \int_{\O_1}(D\Phi)(x-y,t)\cdot\ora{n}(x)\,\psi(y)\,dy+\int_{0}^{t}\int_{\O}(D\Phi)(x-y,t-\tau)\cdot\ora{n}(x)\,f(y,\tau)\,dy\,d\tau \nonumber\\
&+ \int_{0}^{t}\int_{\p\O}(D\Phi)(x-y,t-\tau)\cdot\ora{n}(x)\,\varphi(y,\tau)\,dS_{y}\,d\tau+\frac{1}{2}\varphi(x,t).
\end{align}
Therefore, (\ref{bdry cond., linear prob}) is reduced to for any $(x,t)\in (\Gamma_1\cup\Gamma_2)\times(0,T]$,
\be\label{fixed pt. eq. for bdry fcn}
\varphi(x,t)=\int_{0}^{t}\int_{\p\O}K(x,t;y,\tau)\,\varphi(y,\tau)\,dS_{y}\,d\tau + H(x,t),\ee
where 
\[K(x,t;y,\tau)=-2\Big[(D\Phi)(x-y,t-\tau)\cdot\ora{n}(x)+\beta(x,t)\,\Phi(x-y,t-\tau)\Big]\]
and 
\begin{align*}
H(x,t)= &-2\int_{\O_1}\Big[(D\Phi)(x-y,t)\cdot\ora{n}(x)+\beta(x,t)\,\Phi(x-y,t)\Big]\psi(y)\,dy\\
&-2\int_{0}^{t}\int_{\O}\Big[(D\Phi)(x-y,t-\tau)\cdot\ora{n}(x)+\beta(x,t)\,\Phi(x-y,t-\tau)\Big]f(y,\tau)\,dy\,d\tau\\
&+2\,g(x,t).
\end{align*}
Thus, the proof of (\ref{bdry cond., linear prob}) becomes the search for a fixed point $\varphi\in \mathcal{B}_{T}$ of (\ref{fixed pt. eq. for bdry fcn}).

In the following, we will construct a fixed point $\varphi\in \mathcal{B}_{T}$ of (\ref{fixed pt. eq. for bdry fcn}). By Lemma \ref{Lemma, geom. prop. of C^2 bdry}, we get for any $x,y\in\p\O$, $0\leq \tau<t\leq T$,
\begin{align}\label{est. for the kernel of the induction def for bdry fcn, linear}
|K(x,t;y,\tau)| &\leq C\,\bigg[\frac{1}{(t-\tau)^{n/2}}+\frac{|x-y|^2}{(t-\tau)^{n/2+1}}\bigg]\,e^{-\frac{|x-y|^2}{4(t-\tau)}} \nonumber\\
&\leq C\,(t-\tau)^{-3/4}\,|x-y|^{-(n-3/2)}.
\end{align}
Then using the fact $\psi\in C^{1}(\ol{\O}_1)$ and the integration by parts, we obtain
\be\label{integration by parts for initial fcn}\begin{split}
&\int_{\O_1}(D\Phi)(x-y,t)\cdot\ora{n}(x)\,\psi(y)\,dy=-\int_{\O_1}D_{y}\big[\Phi(x-y,t)\big]\,\psi(y)\,dy\cdot\ora{n}(x)\\
=&-\int_{\p\O_1}\Phi(x-y,t)\,\psi(y)\,\ora{n}(y)\,dy\cdot\ora{n}(x)+\int_{\O_1}\Phi(x-y,t)\,D\psi(y)\,dy\cdot\ora{n}(x). \end{split}\ee
Consequently, as a function in $(x,t)$, 
\[\int_{\O_1}(D\Phi)(x-y,t)\cdot\ora{n}(x)\,\psi(y)\,dy\in C\big(\p\O\times[0,T]\big)\subset\mathcal{B}_{T}.\] 
Then it is readily to check that $H\in\mathcal{B}_{T}$.

Next, we define $\varphi_0(x,t)= H(x,t)$ on $(\Gamma_1\cup\Gamma_2)\times(0,T]$ and for any $j\geq 1$, we define $\varphi_j:(\Gamma_1\cup\Gamma_2)\times(0,T]\rightarrow\m{R}$ by
\be\label{induction def for bdry fcn}
\varphi_{j}(x,t)= \int_{0}^{t}\int_{\p\O}K(x,t;y,\tau)\,\varphi_{j-1}(y,\tau)\,dS_{y}\,d\tau+H(x,t).\ee
Because of (\ref{est. for the kernel of the induction def for bdry fcn}), we can prove by induction that for any $j\geq 0$, $\varphi_{j}$ is well defined and belongs to $\mathcal{B}_{T}$. The next goal is to show that $\varphi_{j}(x,t)$ uniformly converges to some function $\varphi(x,t)$ on $(\Gamma_1\cup\Gamma_2)\times(0,T]$ as $j\rightarrow\infty$, which makes $\varphi$ to be the fixed point of (\ref{fixed pt. eq. for bdry fcn}) in $\mathcal{B}_{T}$. 

To show the uniform convergence of $\{\varphi_j\}_{j\geq 0}$, we define a sequence of functions $\{K_j\}_{j\geq 0}$ on $D_{\O,T}$ as following. For any $(x,t;y,\tau)\in D_{\O,T}$,
\[K_0(x,t;y,\tau)\triangleq K(x,t;y,\tau).\]
For any $j\geq 1$ and $(x,t;y,\tau)\in D_{\O,T}$, 
\be\label{induction def for kernel of fix pt., linear}
K_{j}(x,t;y,\tau)\triangleq \int_{\tau}^{t}\int_{\p\O}K(x,t;z,\sigma)\,K_{j-1}(z,\sigma;y,\tau)\,dS_{z}\,d\sigma. \ee

Based on (\ref{induction def for bdry fcn}) and (\ref{induction def for kernel of fix pt., linear}), again by induction, one can prove that for any $j\geq 1$ and for any $(x,t)\in (\Gamma_1\cup\Gamma_2)\times(0,T]$,
\[\varphi_{j}(x,t)=\varphi_{j-1}(x,t)+\int_{0}^{t}\int_{\p\O}K_{j-1}(x,t;y,\tau)\,H(y,\tau)\,dS_{y}\,d\tau,\]
which implies 
\be\label{explicit formula for varphi_j}
\varphi_{j}(x,t)=\sum_{i=0}^{j-1}\int_{0}^{t}\int_{\p\O}K_{i}(x,t;y,\tau)\,H(y,\tau)\,dS_{y}\,d\tau+H(x,t).\ee
Writing
\be\label{def for accumulated kernel, linear}
K^{*}(x,t;y,\tau)\triangleq \sum_{j=0}^{\infty}K_{j}(x,t;y,\tau), \ee
by Lemma \ref{Lemma, property of iterative kernels}, we know $K^{*}$ is well-defined and $\sum_{j=0}^{\infty}K_j$ converges uniformly to $K^{*}$ on $D_{T,\O}$. Moreover, there exists a constant $C^{*}=C^{*}(n,\O,T)$ such that
for any $(x,t;y,\tau)\in D_{\O,T}$,
\be\label{est for accumulated kernel, linear}
|K^{*}(x,t;y,\tau)|\leq C^{*}(t-\tau)^{-3/4}\,|x-y|^{-(n-3/2)}.\ee
Consequently, it follows from (\ref{explicit formula for varphi_j}) and (\ref{def for accumulated kernel, linear}) that $\varphi_j$ converges uniformly to the function $\varphi$ on $(\Gamma_1\cup\Gamma_2)\times(0,T]$,  where 
\begin{align}\label{explicit formula with accumulated kernel, linear}
\varphi(x,t) \triangleq \int_{0}^{t}\int_{\p\O}K^{*}(x,t;y,\tau)\,H(y,\tau)\,dS_{y}\,d\tau+H(x,t), \quad\forall\, (x,t)\in (\Gamma_1\cup\Gamma_2)\times(0,T].
\end{align}
Thus, $\varphi$ is a fixed point of (\ref{fixed pt. eq. for bdry fcn}) in $\mathcal{B}_{T}$ and therefore the function $u$ defined in (\ref{sol'n for simple model, conjectured form}) satisfies (\ref{bdry cond., linear prob}).

Now as our plan, it only left to confirm this function $u$ satisfies (\ref{bdry cond., interface, linear prob}) as well. Making use of (\ref{sol'n for simple model, conjectured form}), (\ref{Normal Deri. Def.}) and Theorem \ref{Thm, Jump Relation, Adapted}, we get for any $x\in\wt{\Gamma}$, $0<t\leq T$, $\frac{\p u}{\p n}(x,t)$ exists and
\begin{align}\label{normal deri. for interface pts}
\frac{\p u}{\p n}(x,t)=& \int_{\O_1}(D\Phi)(x-y,t)\cdot\ora{n}(x)\,\psi(y)\,dy+\int_{0}^{t}\int_{\O}(D\Phi)(x-y,t-\tau)\cdot\ora{n}(x)\,f(y,\tau)\,dy\,d\tau \nonumber\\
&+ \int_{0}^{t}\int_{\p\O}(D\Phi)(x-y,t-\tau)\cdot\ora{n}(x)\,\varphi(y,\tau)\,dS_{y}\,d\tau+\frac{1}{4}\,\varphi_1(x,t)+\frac{1}{4}\,\varphi_2(x,t).
\end{align}
Then we choose two sequences of points $\{\xi_{k}\}_{k\geq 1}\subset\Gamma_1$ and $\{z_{j}\}_{j\geq 1}\subset\Gamma_2$ which converge to $x$, it follows from (\ref{normal deri. for regular bdry pts}) that
\bes\begin{split}
\frac{\p u}{\p n}(\xi_{k},t)=& \int_{\O_1}(D\Phi)(\xi_{k}-y,t)\cdot\ora{n}(\xi_{k})\,\psi(y)\,dy \\
&+\int_{0}^{t}\int_{\O}(D\Phi)(\xi_{k}-y,t-\tau)\cdot\ora{n}(\xi_{k})\,f(y,\tau)\,dy\,d\tau \\
&+ \int_{0}^{t}\int_{\p\O}(D\Phi)(\xi_{k}-y,t-\tau)\cdot\ora{n}(\xi_{k})\,\varphi(y,\tau)\,dS_{y}\,d\tau+\frac{1}{2}\,\varphi(\xi_{k},t) 
\end{split}\ees
and
\bes\begin{split}
\frac{\p u}{\p n}(z_{j},t)=& \int_{\O_1}(D\Phi)(z_{j}-y,t)\cdot\ora{n}(z_{j})\,\psi(y)\,dy \\
&+\int_{0}^{t}\int_{\O}(D\Phi)(z_{j}-y,t-\tau)\cdot\ora{n}(z_{j})\,f(y,\tau)\,dy\,d\tau \\ 
&+ \int_{0}^{t}\int_{\p\O}(D\Phi)(z_{j}-y,t-\tau)\cdot\ora{n}(z_{j})\,\varphi(y,\tau)\,dS_{y}\,d\tau+\frac{1}{2}\,\varphi(z_{j},t).
\end{split}\ees
Taking $k\rightarrow\infty$ and $j\rightarrow\infty$, we obtain
\be\label{normal deri. for pts on Gamma 1}
\begin{split}
\lim_{k\rightarrow\infty}\frac{\p u}{\p n}(\xi_{k},t)=& \int_{\O_1}(D\Phi)(x-y,t)\cdot\ora{n}(x)\,\psi(y)\,dy \\
&+\int_{0}^{t}\int_{\O}(D\Phi)(x-y,t-\tau)\cdot\ora{n}(x)\,f(y,\tau)\,dy\,d\tau \\
&+ \int_{0}^{t}\int_{\p\O}(D\Phi)(x-y,t-\tau)\cdot\ora{n}(x)\,\varphi(y,\tau)\,dS_{y}\,d\tau+\frac{1}{2}\,\varphi_1(x,t) 
\end{split}\ee
and
\be\label{normal deri. for pts on Gamma 2}
\begin{split}
\lim_{j\rightarrow\infty}\frac{\p u}{\p n}(z_{j},t)=& \int_{\O_1}(D\Phi)(x-y,t)\cdot\ora{n}(x)\,\psi(y)\,dy \\
&+\int_{0}^{t}\int_{\O}(D\Phi)(x-y,t-\tau)\cdot\ora{n}(x)\,f(y,\tau)\,dy\,d\tau \\ 
&+ \int_{0}^{t}\int_{\p\O}(D\Phi)(x-y,t-\tau)\cdot\ora{n}(x)\,\varphi(y,\tau)\,dS_{y}\,d\tau+\frac{1}{2}\,\varphi_2(x,t).
\end{split}\ee
Adding (\ref{normal deri. for pts on Gamma 1}) and (\ref{normal deri. for pts on Gamma 2}) together and noticing (\ref{normal deri. for interface pts}), we attain
\be\label{normal deri. at interface pts is half of the sum of two sides' limits}
\frac{\p u}{\p n}(x,t)=\frac{1}{2}\Big[\lim_{k\rightarrow\infty}\frac{\p u}{\p n}(\xi_{k},t)+\lim_{j\rightarrow\infty}\frac{\p u}{\p n}(z_{j},t)\Big].\ee
Moreover, since $u$ satisfies (\ref{bdry cond., linear prob}), we have
\begin{align*}
\frac{\p u}{\p n}(\xi_{k},t)+\beta(\xi_{k},t)\,u(\xi_{k},t) &=g(\xi_{k},t), \\
\frac{\p u}{\p n}(z_{j},t)+\beta(z_{j},t)\,u(z_{j},t) &=g(z_{j},t).
\end{align*}
Sending $k\rightarrow\infty$ and $j\rightarrow\infty$, we obtain
\begin{align}\label{normal deri. limit on Gamma 1}
\lim_{k\rightarrow\infty}\frac{\p u}{\p n}(\xi_{k},t) &= g_1(x,t)-\beta_1(x,t)\,u(x,t)\\
\label{normal deri. limit on Gamma 2}
\lim_{j\rightarrow\infty}\frac{\p u}{\p n}(z_{j},t) &=
g_2(x,t)-\beta_2(x,t)\,u(x,t)
\end{align}
Combining (\ref{normal deri. at interface pts is half of the sum of two sides' limits}), (\ref{normal deri. limit on Gamma 1}) and (\ref{normal deri. limit on Gamma 2}) together, (\ref{bdry cond., interface, linear prob}) follows. 
\end{proof}

Next, we will prove the comparison principle and the uniqueness of the solution by applying Theorem \ref{Thm, exist. for linear simple model}. But before that, let's prove the following easier comparison result.

\begin{lemma}\label{Lemma, weak comparison for linear prob.}
Suppose in (\ref{Linear Prob, Simple Model}), $f\geq 0$ on $\ol{\O}\times[0,T]$, $\psi>0$ on $\ol{\O}$ and $$\inf\limits_{(\Gamma_1\cup\Gamma_2)\times[0,T]}\,g>0,$$ then the solution $u>0$ on $\ol{\O}\times[0,T]$.
\end{lemma}
\begin{proof}
Since $\psi>0$ on $\ol{\O}$, we have $m\triangleq \min\limits_{\ol{\O}}\psi>0$. Now we claim $u> 0$ on $\ol{\O}\times[0,T]$. If not, then there will exist $x_0\in\ol{\O}$ and $t_0\in(0,T]$ such that 
\[u(x_0,t_0)=0=\min_{\ol{\O}\times[0,t_0]} u.\]
By the strong maximum principle, $x_0\in\p\O$. If $x_0\in\Gamma_1\cup\Gamma_2$, then  
\[0< g(x_0,t_0)=\frac{\p u}{\p n}(x_0,t_0)+\beta(x_0,t_0)\,u(x_0,t_0)=\frac{\p u}{\p n}(x_0,t_0)\leq 0,\]
which is impossible. If $x_0\in\wt{\Gamma}$, then
\begin{align*}
0&<\frac{1}{2}\,\big[g_1(x_0,t_0)+g_2(x_0,t_0)\big]\\
&=\frac{\p u}{\p n}(x_0,t_0)+ \frac{1}{2}\,\big[\beta_1(x_0,t_0)+\beta_2(x_0,t_0)\big]u(x_0,t_0)\\
&=\frac{\p u}{\p n}(x_0,t_0)\leq 0,
\end{align*} 
which is also a contradiction. Thus, the Lemma follows.
\end{proof}

\begin{corollary}\label{Cor, comparison and uniqueness for linear prob. of simple model}
Suppose in (\ref{Linear Prob, Simple Model}), $f\geq 0$ on $\ol{\O}\times[0,T]$, $\psi\geq 0$ on $\ol{\O}$ and
$g\geq 0$ on $(\Gamma_1\cup\Gamma_2)\times(0,T]$, then the solution $u\geq 0$ on $\ol{\O}\times[0,T]$. In particular, the solution to (\ref{Linear Prob, Simple Model}) on $\ol{\O}\times[0,T]$ is unique.
\end{corollary}
\begin{proof}
Due to Theorem \ref{Thm, exist. for linear simple model}, there exists a solution $v\in\mathcal{A}_{T}$ to the following problem:
\bes\left\{\begin{array}{lll}
v_t(x,t)-\Delta v(x,t)=1 &\text{in}& \O\times(0,T], \\
\frac{\p v}{\p n}(x,t)+\beta(x,t) v(x,t)=1 &\text{on}& (\Gamma_1\cup\Gamma_2)\times(0,T],\\
v(x,0)=1 &\text{in}& \O.
\end{array}\right. \ees
For any $\v>0$, we define $w_{\v}=u+\v\,v$, then $w_{\v}$ satisfies
\bes\left\{\begin{array}{lll}
(w_{\v})_t(x,t)-\Delta w_{\v}(x,t)=f+\v\geq\v &\text{in}& \O\times(0,T], \\
\frac{\p w_{\v}}{\p n}(x,t)+\beta(x,t) w_{\v}(x,t)=g+\v\geq\v &\text{on}& (\Gamma_1\cup\Gamma_2)\times(0,T],\\
w_{\v}(x,0)=\psi+\v\geq\v &\text{in}& \O.
\end{array}\right. \ees
By applying Lemma \ref{Lemma, weak comparison for linear prob.}, $w_{\v}\geq 0$ on $\ol{\O}\times[0,T]$. Taking $\v\rightarrow 0$, we get $u\geq 0$ on $\ol{\O}\times[0,T]$. 
\end{proof}

\subsection{Nonlinear Case}
\label{Subsec, Nonlinear Case}
This subsection is devoted to the existence and uniqueness of the solution to the following problem with a local nonlinear Neumann boundary condition:
\be\label{Nonlinear Prob, Simple Model}
\left\{\begin{array}{lll}
u_t(x,t)-\Delta u(x,t)=f(x,t) &\text{in}& \O\times(0,T], \\
\frac{\p u}{\p n}(x,t)=\eta(x)F\big(u(x,t)\big) &\text{on}& \Gamma_1\times(0,T],\\
\frac{\p u}{\p n}(x,t)=0 &\text{on}& \Gamma_2\times(0,T],\\
u(x,0)=\psi(x) &\text{in}& \O,
\end{array}\right. \ee 
where $f\in C^{\alpha,\alpha/2}\big(\ol{\O}\times[0,T]\big)$, $\eta\in C^{1}(\ol{\Gamma}_1)$ and $\eta\geq 0$, $F\in C^{1}(\m{R})$, $\psi\in C^{1}(\ol{\O})$. The solution is understood in the following way.

\begin{definition}\label{Def, soln to Nonlinear Prob.}
For any $T>0$, a solution to (\ref{Nonlinear Prob, Simple Model}) on $\ol{\O}\times[0,T]$ means a function $u$ in $\mathcal{A}_{T}$ that satisfies (\ref{Nonlinear Prob, Simple Model}) pointwise and moreover, for any $(x,t)\in\wt{\Gamma}\times(0,T]$, $\frac{\p u}{\p n}(x,t)$ exists and
\be\label{interface bdry deri. for nonlinear model}
\frac{\p u}{\p n}(x,t)=\frac{1}{2}\,\eta(x)\,F\big(u(x,t)\big). \ee
\end{definition}
This time, we will first show some comparison principles and then discuss the existence of the solution. 

\begin{theorem}\label{Thm, nonlinear comparison and uniqueness}
Suppose $u_{i}\in\mathcal{A}_{T}(i=1,2)$ is the solution to (\ref{Nonlinear Prob, Simple Model}) on $\ol{\O}\times[0,T]$ with right hand side 
$f_{i}(i=1,2)$, $\eta F_{i}(i=1,2)$ and $\psi_{i}(i=1,2)$. If $f_1\geq f_2$ on $\ol{\O}\times[0,T]$, $F_1\geq F_2$ on $\m{R}$, 
$\psi_1\geq\psi_2$ on $\ol{\O}$ and $\eta\geq 0$ on $\ol{\Gamma}_1$, then $u_1\geq u_2$ on $\ol{\O}\times[0,T]$. As a consequence, the solution to (\ref{Nonlinear Prob, Simple Model}) is unique.
\end{theorem}
\begin{proof}
Let $f=f_1-f_2$, $\psi=\psi_1-\psi_2$ and $w=u_1-u_2$, then we have $f\geq 0$ on $\ol{\O}\times[0,T]$, $\psi\geq 0$ on $\ol{\O}$ and 
$F_2\big(u_1(x,t)\big)-F_2\big(u_2(x,t)\big)=\beta(x,t)\,w(x,t)$, where 
\[\beta(x,t)=\int_{0}^{1}F_2'\big(s\,u_1(x,t)+(1-s)u_2(x,t)\big)\,ds.\]
Thus $w$ satisfies the following equations
\[\left\{\begin{array}{lll}
w_t(x,t)-\Delta w(x,t)=f(x,t)\geq 0 &\text{in}& \O\times(0,T], \\
\frac{\p w}{\p n}(x,t)-\eta(x)\,\beta(x,t)\,w(x,t)=\eta(x)\Big[F_1\big(u_1(x,t)\big)-F_2\big(u_1(x,t)\big)\Big]\geq 0 &\text{on}& \Gamma_1\times(0,T],\\
\frac{\p w}{\p n}(x,t)=0 &\text{on}& \Gamma_2\times(0,T],\\
w(x,0)=\psi(x)\geq 0 &\text{in}& \O.
\end{array}\right.\]
Now it follows from Corollary \ref{Cor, comparison and uniqueness for linear prob. of simple model} that $w\geq 0$. \end{proof}

\begin{theorem}\label{Thm, positivity of the solution}
Suppose $u\in\mathcal{A}_T$ is the solution to (\ref{Nonlinear Prob, Simple Model}) with $f\geq 0$ on $\ol{\O}\times[0,T]$, $\psi\geq 0$ on $\ol{\O}$, $\eta\geq 0$ on $\ol{\Gamma}_1$ and $F(0)\geq 0$, then $u\geq 0$ on $\ol{\O}\times[0,T]$. In addition, if $\psi\not\equiv 0$, then 
$u(x,t)>0,\,\forall \, x\in\ol{\O},\,0<t\leq T$.
\end{theorem}
\begin{proof}
To prove the first statement, we write $$F\big(u(x,t)\big)=F\big(u(x,t)\big)-F(0)+F(0)=\beta(x,t)u(x,t)+F(0),$$
where \[\beta(x,t)=\int_{0}^{1}F'\big(su(x,t)\big)\,ds.\]
Hence $u$ satisfies 
\[\left\{\begin{array}{lll}
u_t(x,t)-\Delta u(x,t)=f(x,t)\geq 0 &\text{in}& \O\times(0,T], \\
\frac{\p u}{\p n}(x,t)-\eta(x)\,\beta(x,t)\,u(x,t)=\eta(x)\,F(0)\geq 0 &\text{on}& \Gamma_1\times(0,T],\\
\frac{\p u}{\p n}(x,t)=0 &\text{on}& \Gamma_2\times(0,T],\\
u(x,0)=\psi(x)\geq 0 &\text{in}& \O.
\end{array}\right.\]
It then follows from Corollary \ref{Cor, comparison and uniqueness for linear prob. of simple model} that $u\geq 0$. Now in order to prove the second statement, we suppose additionally that $\psi\not\equiv 0$, then by applying the strong maximum principle and the Hopf lemma, we get $u(x,t)>0,\,\forall \, x\in\ol{\O},\,0<t\leq T$. \end{proof}

\begin{corollary}\label{Cor, comparison for cut-off fcn.}
Suppose $u_{i}\in\mathcal{A}_{T}(i=1,2)$ is the solution to (\ref{Nonlinear Prob, Simple Model}) on $\ol{\O}\times[0,T]$ with right hand side $f$, $\eta_{i}F$ and $\psi$. If $f\geq 0$ on $\ol{\O}\times[0,T]$, $\psi\geq 0$ on $\ol{\O}$, $F(t)\geq 0$ for $t\geq 0$ and $\eta_1\geq \eta_2\geq 0$ on $\ol{\Gamma}_1$, then $u_1\geq u_2$ on $\ol{\O}\times[0,T]$. 
\end{corollary}
\begin{proof}
Firstly, by Theorem \ref{Thm, positivity of the solution}, $u_2$ is nonnegative on $\ol{\O}\times[0,T]$ and therefore $F\circ u_2$ is nonnegative on $\ol{\O}\times[0,T]$. Writing $w=u_1-u_2$, then for any $x\in\Gamma_1$, $t\in(0,T]$,  
\begin{align*}
&\quad\,\eta_1(x)F\big(u_1(x,t)\big)-\eta_2(x)F\big(u_2(x,t)\big)\\
&=\eta_1(x)\big[F\big(u_1(x,t)\big)-F\big(u_2(x,t)\big)\big]+F\big(u_2(x,t)\big)\big[\eta_1(x)-\eta_2(x)\big]\\
&\geq \eta_1(x)\big[F\big(u_1(x,t)\big)-F\big(u_2(x,t)\big)\big]\\
&=\eta_1(x)\beta(x,t)w(x,t),
\end{align*}
where \[\beta(x,t)=\int_{0}^{1}F'\big(s\,u_1(x,t)+(1-s)u_2(x,t)\big)\,ds.\]
Thus, $w$ satisfies 
\[\left\{\begin{array}{lll}
w_t(x,t)-\Delta w(x,t)=0 &\text{in}& \O\times(0,T], \\
\frac{\p w}{\p n}(x,t)-\eta_1(x)\,\beta(x,t)\,w(x,t)\geq 0 &\text{on}& \Gamma_1\times(0,T],\\
\frac{\p w}{\p n}(x,t)=0 &\text{on}& \Gamma_2\times(0,T],\\
w(x,0)=0&\text{in}& \O.
\end{array}\right.\]
Applying Corollary \ref{Cor, comparison and uniqueness for linear prob. of simple model}, we have $w\geq 0$ on $\ol{\O}\times[0,T]$.
\end{proof}

Next, we turn to the existence of the solution. As a common process to deal with the nonlinear problem, we will take advantage of the theories for the linear problems and some fixed point theorems. Let $T>0$, $R>0$ and $X_T=C\big(\ol{\O}\times[0,T]\big)$ be equipped with the maximum norm: $||u||\triangleq \max\limits_{\ol{\O}\times[0,T]} |u|$ for any $u\in X_T$, then $X_T$ is a Banach space and
\[X_{T,R}\triangleq\{v\in X_T: ||v||\leq R\}\]  
is also a Banach space. For any $v\in X_{T,R}$, it follows from Theorem \ref{Thm, exist. for linear simple model} and Corollary \ref{Cor, comparison and uniqueness for linear prob. of simple model} that there exists a unique solution $u\in\mathcal{A}_T$ to the following problem
\be\label{fixed pt mapping}
\left\{\begin{array}{lll}
u_t(x,t)-\Delta u(x,t)=f(x,t) &\text{in}& \O\times(0,T], \\
\frac{\p u}{\p n}(x,t)=\eta(x)F\big(v(x,t)\big) &\text{on}& \Gamma_1\times(0,T],\\
\frac{\p u}{\p n}(x,t)=0 &\text{on}& \Gamma_2\times(0,T],\\
u(x,0)=\psi(x) &\text{in}& \O.
\end{array}\right.\ee
Thus, it determines a mapping $\Psi_T:X_{T,R}\rightarrow \mathcal{A}_{T}$. Our strategy is to pick up a suitable $R$ (depending on $T$) show that $\Psi_T$ has a fixed point in $X_{T,R}$, which turns out to be the unique solution to (\ref{Nonlinear Prob, Simple Model}). 

In the proof of Theorem \ref{Thm, exist. for nonlin. simple model}, we will utilize the Schauder fixed point theorem, which requires to verify the following three things:
\begin{itemize}
\item[(i)] $\Psi_T$ maps $X_{T,R}$ to $\mathcal{A}_{T}\cap X_{T,R}$ for some suitably chosen $R>0$;
\item[(ii)] $\Psi_T:X_{T,R}\rightarrow X_{T,R}$ is continuous;
\item[(iii)] $\Psi_T(X_{T,R})$ is precompact in $X_{T,R}$.
\end{itemize}
Usually, the requirement (iii) is the most technical part and this time it requires the following Lemma \ref{Lemma, equicontinuity}, which is a fact mentioned in the proof of (\cite{Friedman}, Theorem 13, Sec. 5, Chap. 7).

\begin{lemma}\label{Lemma, equicontinuity}
Given $T>0$ and $\{\varphi_{j}\}_{j\geq 1}\subset L^{\infty}\big((\Gamma_1\cup\Gamma_2)\times(0,T]\big)$, we define 
\be\label{w_j}
w_{j}(x,t)=\int_{0}^{t}\int_{\p\O}\Phi(x-y,t-\tau)\,\varphi_{j}(y,\tau)\,dS_{y}\,d\tau,\quad\forall\, (x,t)\in\ol{\O}\times[0,T].\ee
If $\{\varphi_{j}\}_{j\geq 1}$ are uniformly bounded on $(\Gamma_1\cup\Gamma_2)\times(0,T]$, then $\{w_j\}_{j\geq 1}$ are uniformly bounded and equicontinuous on $\ol{\O}\times[0,T]$.
\end{lemma}

\begin{proof}
By using the estimate
\[|\Phi(x-y,t-\tau)|\leq C\,(t-\tau)^{-3/4}\,|x-y|^{-(n-3/2)},\] 
it is not hard to prove this Lemma, so we omit the proof here.
\end{proof}

Now based on the arguments in (\cite{Friedman}, Theorem 13, Sec. 5, Chap. 7) and (\cite{L-GMW}, Theorem 1.3), we conclude the following theorem on the existence of the solution.

\begin{theorem}\label{Thm, exist. for nonlin. simple model}
For the nonlinear problem (\ref{Nonlinear Prob, Simple Model}) with $f$, $\eta$, $F$, $\psi$ described there, we have the following two conclusions.
\begin{itemize}
\item[(1)] There exists $T_0>0$ such that for any $0<T\leq T_0$, there exists a unique solution $u\in\mathcal{A}_{T}$ to (\ref{Nonlinear Prob, Simple Model}) on $\ol{\O}\times[0,T]$.
\item[(2)] If $F$ is a bounded function on $\m{R}$, then for any $T>0$, there exists a unique solution $u\in\mathcal{A}_{T}$ to (\ref{Nonlinear Prob, Simple Model}) on $\ol{\O}\times[0,T]$.
\end{itemize}
\end{theorem}

\begin{proof}
Just as the heuristic idea before Lemma \ref{Lemma, equicontinuity}, in order to prove the existence of a solution, we will use Schauder fixed point theorem to show $\Psi_{T}$ has a fixed point in $X_{T,R}$ for some $R>0$. Namely, we need to verify the following three requirements:
\begin{itemize}
\item[(i)] $\Psi_T$ maps $X_{T,R}$ to $\mathcal{A}_{T}\cap X_{T,R}$ for some suitably chosen $R>0$ (depending on $T$);
\item[(ii)] $\Psi_T:X_{T,R}\rightarrow X_{T,R}$ is continuous;
\item[(iii)] $\Psi_T(X_{T,R})$ is precompact in $X_{T,R}$.
\end{itemize}
In the following, we will prove (1) and (2) in Theorem \ref{Thm, exist. for nonlin. simple model} together. Actually, the proofs of requirements (ii) and (iii) for (1) and (2) are identically the same, only the proofs of requirement (i) has slightly difference.

Firstly, given $T>0$, let's recall how we construct $u\triangleq \Psi_T(v)$ for $v\in X_{T,R}$. We will use the same notations as those in the proof of Theorem \ref{Thm, exist. for linear simple model}, but with $\beta=0$ and $g(x,t)=\eta(x)\,F\big(v(x,t)\big)\chifcn_{\Gamma_1}(x)$, where 
\[\chifcn_{\Gamma_1}(x)\triangleq \left\{\begin{array}{cc}
1 & x\in\Gamma_1 \\
0 & x\notin\Gamma_1
\end{array}\right.\]
is the indicator function. Thus $u$ has the following expression: for any $(x,t)\in\ol{\O}\times[0,T]$,
\be\label{sol'n for nonlinear simple model, conjectured form}\begin{split}
u(x,t)\triangleq & \int_{\O_1}\Phi(x-y,t)\,\psi(y)\,dy+\int_{0}^{t}\int_{\O}\Phi(x-y,t-\tau)\,f(y,\tau)\,dy\,d\tau\\
&+\int_{0}^{t}\int_{\p\O}\Phi(x-y,t-\tau)\,\varphi(y,\tau)\,dS_{y}\,d\tau.
\end{split}\ee 
Here $\varphi\in \mathcal{B}_{T}$ satisfies for any $(x,t)\in (\Gamma_1\cup\Gamma_2)\times(0,T]$,
\bes
\varphi(x,t)=\int_{0}^{t}\int_{\p\O}K(x,t;y,\tau)\,\varphi(y,\tau)\,dS_{y}\,d\tau + H(x,t),\ees
where
\be\label{kernel for nonlinear eq.}
K(x,t;y,\tau)=-2\,(D\Phi)(x-y,t-\tau)\cdot\ora{n}(x) \ee
and 
\be\label{remainder for nonlinear eq.} \begin{split}
H(x,t)= &-2\int_{\O_1}(D\Phi)(x-y,t)\cdot\ora{n}(x)\,\psi(y)\,dy\\
&-2\int_{0}^{t}\int_{\O}(D\Phi)(x-y,t-\tau)\cdot\ora{n}(x)\,f(y,\tau)\,dy\,d\tau\\
&+2\,\eta(x)\,F\big(v(x,t)\big)\,\chifcn_{\Gamma_1}(x).
\end{split}\ee
Because the function $K$ in (\ref{kernel for nonlinear eq.}) also satisfies (\ref{est. for the kernel of the induction def for bdry fcn}), we can apply Lemma \ref{Lemma, property of iterative kernels} and follow the same way as the derivations of (\ref{est for accumulated kernel, linear}), (\ref{explicit formula with accumulated kernel, linear})  to get
\be\label{explicit formula for bdry fcn, nonlinear}
\varphi(x,t)=\int_{0}^{t}\int_{\p\O}K^{*}(x,t;y,\tau)\,H(y,\tau)\,dS_{y}\,d\tau + H(x,t) \ee
for some function $K^{*}$. Moreover, there exists a constant $C^{*}=C^{*}(n,\O,T)$ such that
\be\label{bound for accumulate kernel, point. conv., nonlinear}
|K^{*}(x,t;y,\tau)|\leq C^{*}(t-\tau)^{-3/4}\,|x-y|^{-(n-3/2)}. \ee

Next, we will first assume requirement (i) and prove requirements (ii) and (iii), then we will confirm requirement (i) for the Cases (1) and (2) in Theorem \ref{Thm, exist. for nonlin. simple model} respectively. Given $T>0$, we assume there exists $R>0$ such that $\Psi_T:X_{T,R}\rightarrow \mathcal{A}_{T}\cap X_{T,R}$. Define $M_{F}, M_{F'}:[0,\infty)\rightarrow\m{R}$ by 
\[M_{F}(r)=\sup_{|x|\leq r}|F(x)|\]
and 
\[M_{F'}(r)=\sup_{|x|\leq r}|F'(x)|,\]
then both $M_{F}(r)$ and $M_{F'}(r)$ are finite for any $r\geq 0$, since $F\in C^{1}(\m{R})$. In the following, for any $v\in X_{T,R}$, we write $H$, $\varphi$ and $u$ as defined in (\ref{remainder for nonlinear eq.}), (\ref{kernel for nonlinear eq.}) and (\ref{sol'n for nonlinear simple model, conjectured form}) respectively. For any $v_j\in X_{T,R}$ ($j\geq 1$), we analogously write $H_j$, $\varphi_j$ and $u_j$. 

\begin{itemize}
\item Proof of Requirement (ii). Given $\{v_{j}\}_{j\geq 1}\subset X_{T,R}$ and $v_j\rightarrow v$ in $X_{T,R}$, we want to show $\Psi_T(v_j)\rightarrow \Psi_T(v)$ in $X_{T,R}$. Because $v$ and all the $v_j (j\geq 1)$ belong to $X_{T,R}$, then for any $(x,t)\in\ol{\O}\times[0,T]$, $|v(x,t)|\leq R$ and $|v_{j}(x,t)|\leq R$. Thus, by the mean value theorem and the fact $M_{F'}(R)<\infty$, it follows from (\ref{remainder for nonlinear eq.}) that $H_j\rightrightarrows H$ on $(\Gamma_1\cup\Gamma_2)\times(0,T]$ (here $``\rightrightarrows"$ means ``converges uniformly to''). Then by (\ref{explicit formula for bdry fcn, nonlinear}) and (\ref{bound for accumulate kernel, point. conv., nonlinear}), $\varphi_{j}\rightrightarrows\varphi$ on $(\Gamma_1\cup\Gamma_2)\times (0,T]$. Finally, due to the expression (\ref{sol'n for nonlinear simple model, conjectured form}), we have $u_j\rightrightarrows u$ on $\ol{\O}\times[0,T]$, which implies $u\in X_{T,R}$ and $\Psi_{T}(v_j)\rightarrow \Psi_{T}(v)$ in $X_{T,R}$.

\item Proof of Requirement (iii). In this proof, we will use $C$ to denote a constant which is independent of $j$, $x$ and $t$, but may depend on $n$, $\O$, $\O_1$, $T$, $R$, $M_F(R)$, $\sup |f|$, $\sup |\psi|$, $\sup |D\psi|$ and $\sup |\eta|$. $C$ may be different from line to line. Given any sequence $\{v_j\}_{j\geq 1}\subset X_{T,R}$, we want to show $\{\Psi_{T}(v_j)\}_{j\geq 1}$ has a subsequence which converges to some function $u$ in $X_{T,R}$. Since $v_j\in X_{T,R}$ for any $j\geq 1$, then for any $j\geq 1$ and for any $(x,t)\in\ol{\O}\times[0,T]$, $|v_j(x,t)|\leq R$. Recalling (\ref{integration by parts for initial fcn}), we know 
\[\Big|\int_{\O_1}(D\Phi)(x-y,t)\cdot\ora{n}(x)\,\psi(y)\,dy\Big|\]
is bounded by some constant $C$. As a result, based on (\ref{remainder for nonlinear eq.}), these exists another constant $C$ such that for any $j\geq 1$ and for any $(x,t)\in\ol{\O}\times[0,T]$, 
\[|H_j(x,t)|\leq C.\]
Then due to (\ref{explicit formula for bdry fcn, nonlinear}) and (\ref{bound for accumulate kernel, point. conv., nonlinear}), there exists some constant $C$ such that for any $j\geq 1$ and for any $(x,t)\in\ol{\O}\times[0,T]$, \[|\varphi_j(x,t)|\leq C.\]
Now using (\ref{sol'n for nonlinear simple model, conjectured form}) and Lemma \ref{Lemma, equicontinuity}, we find $\{u_j\}_{j\geq 1}$ is uniformly bounded and equicontinuous on $\ol{\O}\times[0,T]$. Hence, it follows from the Arzela-Ascoli theorem that $\{u_j\}_{j\geq 1}$ has a subsequence $\{u_{j_k}\}_{k\geq 1}$ which converges uniformly  to some function $u$ on $\ol{\O}\times[0,T]$. Since $u_{j_k}\in X_{T,R}$, it is readily to see that $u$ is also in $X_{T,R}$. Thus, $\Psi_T(X_{T,R})$ is precompact in $X_{T,R}$.
\end{itemize}

Now we turn to verify Requirement (i). 
\begin{itemize}
\item Proof of Requirement (i) for (1). We will find $0<T_0\leq 1$ such that for any $0<T\leq T_0$, there exists $R>0$ such that $\Psi_T$ maps $X_{T,R}$ to $\mathcal{A}_{T}\cap X_{T,R}$. In this proof, $C$ will denote a constant which is independent of $x$, $t$, $R$ and $T$, but may depend on $n$, $\O$, $\O_1$, $\sup |f|$, $\sup |\psi|$, $\sup |D\psi|$ and $\sup |\eta|$. $C$ may be different from line to line. For the first term of (\ref{remainder for nonlinear eq.}), we recall (\ref{integration by parts for initial fcn}) again to get for any $(x,t)\in (\Gamma_1\cup\Gamma_2)\times(0,T]$,
\begin{align}\label{bound for first term}
&\Big|\int_{\O_1}(D\Phi)(x-y,t)\cdot\ora{n}(x)\,\psi(y)\,dy\Big| \nonumber\\
\leq &\Big|\int_{\p\O_1}\Phi(x-y,t)\,\psi(y)\,\ora{n}(y)\,dy\cdot\ora{n}(x)\Big|+\Big|\int_{\O_1}\Phi(x-y,t)\,D\psi(y)\,dy\cdot\ora{n}(x)\Big| \nonumber\\
\leq & C\,\int_{\p\O_1}|x-y|^{-n}\,dy + C\leq C.
\end{align}
For the second term of (\ref{remainder for nonlinear eq.}), we have for any $(x,t)\in (\Gamma_1\cup\Gamma_2)\times(0,T]$,
\begin{align}\label{bound for second term}
&\Big|\int_{0}^{t}\int_{\O}(D\Phi)(x-y,t-\tau)\cdot\ora{n}(x)\,f(y,\tau)\,dy\,d\tau\Big| \nonumber\\
\leq\, & C\,\sup |f|\,\int_{0}^{t}\int_{\O}(t-\tau)^{-3/4}\,|x-y|^{-(n-1/2)}\,dS_{y}\,d\tau \nonumber\\
\leq\, & C\,t^{1/4}\leq C\,T^{1/4}\leq C\,T_{0}^{1/4}\leq C,
\end{align}
where the last inequality uses the assumption $T_{0}\leq 1$. Then it follows from (\ref{bound for first term}), (\ref{bound for second term}) and (\ref{remainder for nonlinear eq.}) that for any $(x,t)\in (\Gamma_1\cup\Gamma_2)\times(0,T]$,
\be\label{bound for H, nonlinear}
|H(x,t)|\leq C+C\,M_{F}(R).\ee
Although the constant $C^{*}$ in (\ref{bound for accumulate kernel, point. conv., nonlinear}) depends on $T$, if one checks its proof, it is readily to see that $C^{*}$ is an increasing function in $T$. As a result, when $T$ is bounded by $1$, $C^{*}$ will also be bounded by some constant $C$, which only depends on $n$ and $\O$. Based on this observation and (\ref{explicit formula for bdry fcn, nonlinear}), we get for any $(x,t)\in (\Gamma_1\cup\Gamma_2)\times(0,T]$,
\begin{align}\label{bound for varphi, nonlinear, local}
|\varphi(x,t)| &\leq C^{*}\,[C+C\,M_{F}(R)]\int_{0}^{t}\int_{\p\O}(t-\tau)^{-3/4}\,|x-y|^{-(n-3/2)}\,dS_{y}\,d\tau+C+C\,M_{F}(R) \nonumber \\
&\leq [C+C\,M_{F}(R)]\,T^{1/4}+C+C\,M_{F}(R) \nonumber \\
&\leq C+C\,M_{F}(R).
\end{align}
Now by (\ref{sol'n for nonlinear simple model, conjectured form}) and (\ref{bound for varphi, nonlinear, local}), we obtain for any $(x,t)\in\ol{\O}\times[0,T]$,
\begin{align}\label{bound for sol'n, nonlinear}
|u(x,t)| &\leq \sup |\psi|+t\,\sup |f|+ C\,\sup |\varphi|\,\int_{0}^{t}\int_{\p\O}(t-\tau)^{-3/4}\,|x-y|^{-(n-3/2)}\,dS_{y}\,d\tau \nonumber \\
&\leq C+\big[C+C\,M_{F}(R)\big]\,C\,T^{1/4} \nonumber \\
&\leq C+C\,M_{F}(R)\,T_{0}^{1/4} \nonumber \\
&\triangleq C_1+C_1\,M_{F}(R)\,T_{0}^{1/4}.
\end{align}
Hence, if we choose $R$ and $T_{0}\leq 1$ such that 
\be\label{choice of R and T_0}
R=2C_1\quad \text{and} \quad T_{0}^{1/4}\,M_{F}(2C_1)<1,\ee
then we have $||u||\leq R$ and therefore $u\triangleq\Psi_T(v)\in X_{T,R}$.  

\item Proof of Requirement (i) for (2). We will prove that for any $T>0$, there exists $R>0$ such that $\Psi_T$ maps $X_{T,R}$ to $\mathcal{A}_{T}\cap X_{T,R}$. From the assumption, $F$ is a bounded function in $\m{R}$, so $\sup_{\m{R}}|F|<\infty$. In the rest of this proof, we will use $C$ to denote a constant just like that in the proof for (1) but additionally allowing $C$ to depend on $T$ and $\sup_{\m{R}}|F|$. As the same derivations as (\ref{bound for first term}), (\ref{bound for second term}) and (\ref{bound for H, nonlinear}), we attain for any $(x,t)\in (\Gamma_1\cup\Gamma_2)\times(0,T]$,
\begin{align}\label{bound for H, nonlinear, global}
|H(x,t)| &\leq C+C\,M_{F}(R) \nonumber\\
&\leq C+C\sup_{\m{R}}|F|=C.
\end{align}
Then based on (\ref{explicit formula for bdry fcn, nonlinear}) and (\ref{bound for H, nonlinear, global}), we get for any $(x,t)\in (\Gamma_1\cup\Gamma_2)\times(0,T]$,
\begin{align*}
|\varphi(x,t)| &\leq C^{*}\,C\,\int_{0}^{t}\int_{\p\O}(t-\tau)^{-3/4}\,|x-y|^{-(n-3/2)}\,dS_{y}\,d\tau+C \nonumber \\
&\leq C\,T^{1/4}+C \nonumber \\
&= C.
\end{align*}
As a result, by (\ref{sol'n for nonlinear simple model, conjectured form}) and (\ref{bound for varphi, nonlinear, local}) again, we obtain for any $(x,t)\in\ol{\O}\times[0,T]$,
\begin{align*}
|u(x,t)| &\leq \sup |\psi|+t\,\sup |f|+ C\,\sup |\varphi|\,\int_{0}^{t}\int_{\p\O}(t-\tau)^{-3/4}\,|x-y|^{-(n-3/2)}\,dS_{y}\,d\tau \\
&\leq C+C\,T^{1/4} \\
&\triangleq C_2. 
\end{align*}
Thus, as long as choosing $R>C_2$, we will have $||u||\leq R$ and consequently $u\triangleq\Psi_T(v)\in X_{T,R}$. 

\end{itemize}

\end{proof}

As we can see from Theorem \ref{Thm, nonlinear comparison and uniqueness}, the solution to (\ref{Nonlinear Prob, Simple Model}) is proved to be global only under the assumption that the function $F$ being bounded on $\m{R}$. Thus, when $F$ is unbounded, we need to consider the maximal solution and figure out when the solution can exist globally.

\begin{definition}\label{Def, maximal sol'n, Nonlin. Prob.}
We call \[T^{*}\triangleq \sup\{T\geq 0:\,\text{there exsits a solution to (\ref{Nonlinear Prob, Simple Model}) on}\,\,\, \ol{\O}\times[0,T]\}\]
to be the maximal existence time for (\ref{Nonlinear Prob, Simple Model}). Moreover, a function $u^{*}\in C^{2,1}\big(\O\times(0,T^{*})\big)\cap C\big(\ol{\O}\times[0,T^{*})\big)$ is called a maximal solution if it solves (\ref{Nonlinear Prob, Simple Model}) on $\ol{\O}\times[0,T]$ for any $T\in(0,T^{*})$.
\end{definition}

\begin{remark}\label{Remark, maximal soln}
It follows from Theorem \ref{Thm, exist. for nonlin. simple model} and Theorem \ref{Thm, nonlinear comparison and uniqueness} that the maximal solution exists and is unique.
\end{remark}

In (\cite{L-GMW}, Corollary 1.1), it concludes that when $T^{*}$ is finite, then it coincides with the blow-up time of the $L^{\infty}$ norm of $u^{*}$. Actually, it is the same situation here. More precisely, we have the following Theorem \ref{Thm, infinity norm characterizes the blow-up}.

\begin{theorem}\label{Thm, infinity norm characterizes the blow-up}
Let $T^{*}$ be the maximal existence time for (\ref{Nonlinear Prob, Simple Model}) and $u^{*}$ be the maximal solution. If $T^{*}<\infty$, then 
\be\sup_{\ol{\O}\times[0,T^{*})}|u^{*}(x,t)|=\infty.\ee
\end{theorem}
\begin{proof}
We refer the readers to the arguments in (\cite{L-GMW}, Corollary 1.1).
\end{proof}

Thus, if we can prove the solution to be bounded all the time, then it exists globally. Moreover, in order to estimate $T^{*}$, one only needs to study the blow-up time for the $L^{\infty}$ norm of the solution. 

\begin{remark}\label{Remark, app. to exist. and uniq.}
As a particular application of the theories established in this section, one can apply Theorem \ref{Thm, exist. for nonlin. simple model}, Theorem \ref{Thm, nonlinear comparison and uniqueness}, Theorem \ref{Thm, positivity of the solution} and Remark \ref{Remark, maximal soln} with $f=0$, $\eta=1$, $F(\lambda)=\lambda^{q}$ and $\psi=u_0$ to our targeted problem (\ref{Prob, Simple Model}) to obtain Theorem \ref{Thm, fundamental theorem}.
\end{remark}

\end{appendix}

\section*{Acknowledgements}
This work was partially supported by NSF grant DMS 0968360 and the Dissertation Completion Fellowship from the Michigan State University.

\end{document}